\documentclass[12pt]{amsart}
\usepackage{}

\usepackage{amsmath}
\usepackage{amsfonts}
\usepackage{amssymb}
\usepackage[all]{xy}           

\usepackage{bbding}
\usepackage{txfonts}
\usepackage{amscd}

\usepackage[shortlabels]{enumitem}
\usepackage{ifpdf}
\ifpdf
  \usepackage[colorlinks,final,backref=page,hyperindex]{hyperref}
\else
  \usepackage[colorlinks,final,backref=page,hyperindex,hypertex]{hyperref}
\fi
\usepackage{tikz}
\usepackage[active]{srcltx}

\makeatletter

\newtheorem{thm}{Theorem}[section]
\newtheorem{lem}[thm]{Lemma}
\newtheorem{cor}[thm]{Corollary}
\newtheorem{pro}[thm]{Proposition}
\newtheorem{ex}[thm]{Example}
\newtheorem{rmk}[thm]{Remark}
\newtheorem{defi}[thm]{Definition}

\newcommand {\emptycomment}[1]{}

\newcommand{\lon }{\,\rightarrow\,}
\newcommand{\be }{\begin{equation}}
\newcommand{\ee }{\end{equation}}

\newcommand{\g}{\mathfrak g}
\newcommand{\h}{\mathfrak h}

\newcommand{\huaS}{\mathcal{S}}

\newcommand{\huaG}{\mathcal{G}}

\newcommand{\huaO}{{\mathcal{O}}}

\newcommand{\frke}{\mathfrak e}

\newcommand{\half}{\frac{1}{2}}

\newcommand{\Id}{{\rm{Id}}}

\newcommand{\br}[1]{   [ \cdot,    \cdot  ]   }

\newcommand{\Hom}{\mathrm{Hom}}

\newcommand{\gl}{\mathfrak {gl}}

\newcommand{\Op}{$\mathcal O$-operator}
\newcommand{\Ops}{$\mathcal O$-operators}

\newcommand{\End}{\mathrm{End}}
\newcommand{\ad}{\mathrm{ad}}

\newcommand{\K}{\mathbb{K}}

\newcommand{\U}{\mathrm{U}}

\newcommand{\Li}{\mathsf{3Lie}}

\begin{document}

\title{Twilled 3-Lie algebras, generalized matched pairs of 3-Lie algebras and $\huaO$-operators}

\author{Shuai Hou}
\address{Department of Mathematics, Jilin University, Changchun 130012, Jilin, China}
\email{hshuaisun@163.com}

\author{Yunhe Sheng}
\address{Department of Mathematics, Jilin University, Changchun 130012, Jilin, China}
\email{shengyh@jlu.edu.cn}

\author{Rong Tang}
\address{Department of Mathematics, Jilin University, Changchun 130012, Jilin, China}
\email{tangrong@jlu.edu.cn}


\begin{abstract}
In this paper, first we introduce the notion of a twilled 3-Lie algebra, and construct an $L_\infty$-algebra, whose Maurer-Cartan elements give rise to new twilled 3-Lie algebras by twisting. In particular, we recover the   Lie $3$-algebra  whose Maurer-Cartan elements are $\huaO$-operators (also called relative Rota-Baxter operators) on 3-Lie algebras. Then we introduce the notion of   generalized matched pairs of 3-Lie algebras using generalized representations of 3-Lie algebras, which will give rise to   twilled 3-Lie algebras. The usual matched pairs of 3-Lie algebras correspond to a special class of twilled 3-Lie algebras, which we call strict twilled 3-Lie algebras.  Finally, we use $\huaO$-operators to construct explicit twilled 3-Lie algebras, and explain why an $r$-matrix for a 3-Lie algebra can not give rise to a double construction 3-Lie bialgebra. Examples of twilled 3-Lie algebras are given to illustrate the various interesting phenomenon.
\end{abstract}



\maketitle

\tableofcontents

\allowdisplaybreaks


\section{Introduction}

 3-Lie
algebras and more generally, $n$-Lie
algebras (also called Filippov algebras)~\cite{Filippov}, have attracted attention
from both mathematics and physics. The $n$-Lie algebra is the algebraic
structure corresponding to Nambu mechanics \cite{N}. See the review article \cite{review,Makhlouf} for more details. The notion of an $\huaO$-operator on a 3-Lie algebra with respect to a representation was introduced in \cite{BGS-3-Bialgebras} to construct solutions of the classical 3-Lie Yang-Baxter equation. When the representation is the adjoint representation, one obtain Rota-Baxter operators on 3-Lie algebras introduced in \cite{BaiRGuo}. In \cite{THS}, a Lie 3-algebra was constructed to characterize $\huaO$-operators on 3-Lie algebras as Maurer-Cartan elements. Moreover, cohomologies and deformations of $\huaO$-operators on 3-Lie algebras were also studied there.

Various twilled algebras were studied in the literature. Roughly speaking, a twilled algebra is an algebra that decomposed into two subalgebras (as direct sum of vector spaces). Twilled Lie algebras were studied in  \cite{Kosmann,Kosmann-1,Lu} with applications to Lie bialgebras and  integrable  systems.  Twilled associative algebras were studied in \cite{CGM,Uchino} with connections to quantization and Rota-Baxter type operators. The notion of a twilled Leibniz algebra was introduced in  \cite{ST} in the study of Leibniz bialgebras.  Twilled algebras are closely related with the twisting theory \cite{Drinfeld,Uchino}. One common property for the above twilled algebras is that they are equivalent to matched pairs. As we will see, this property does not hold anymore for 3-Lie algebras.


In this paper, we introduce the notion of a twilled 3-Lie algebra, which is a 3-Lie algebra that being the direct sum of   subspaces underlying two subalgebras. This condition is also equivalent to the existence of a product structure on the 3-Lie algebra \cite{Sheng-Tang}. We construct an $L_\infty$-algebra from a twilled 3-Lie algebra, and recover the Lie 3-algebra given in \cite{THS} when one consider the semidirect product 3-Lie algebra. Different from the cases of twilled associative algebras, twilled Lie algebras and twilled Leibniz algebras, a twilled 3-Lie algebra is not equivalent to a (usual) matched pair of 3-Lie algebras anymore. This force us to define a strict twilled 3-Lie algebra, which is equivalent to a matched pair of 3-Lie algebra. This also motivate us to define a generalized matched pair of 3-Lie algebras using the generalized representations introduced in \cite{Jiefeng}. It follows that a generalized matched pair of 3-Lie algebras gives rise to a twilled 3-Lie algebra. However, the converse is still not true. We also study the twisting theory of twilled 3-Lie algebras. We show that the twisting of a twilled 3-Lie algebra by a Maurer-Cartan element of the aforementioned  $L_\infty$-algebra is still a twilled 3-Lie algebra. This allows us to construct explicit twilled 3-Lie algebras using $\huaO$-operators on 3-Lie algebras.

The bialgebra theory for 3-Lie algebras was developed  using different approaches \cite{BGS-3-Bialgebras,BRP,BRP-1,HBS,RS}.
  In particular, two kinds of bialgebras, namely a local cocycle 3-Lie bialgebra  and a double construction 3-Lie bialgebra  were introduced in \cite{BGS-3-Bialgebras}, and an $r$-matrix for a 3-Lie algebra gives rise to the former, while the latter is equivalent to a Manin triple (a special matched pair) of 3-Lie algebras. However, it is still unclear there why an $r$-matrix for a 3-Lie algebra can not give rise to a double construction 3-Lie bialgebra. The study in this paper answer this question intrinsically. This is because as aforementioned, different from the binary case, a twilled 3-Lie algebra  is not equivalent to a matched pair of 3-Lie algebras anymore, while  as a special $\huaO$-operator, an  $r$-matrix can only give  rise to a twilled 3-Lie algebra. This suggest that there might be a more general   bialgebra theory for 3-Lie algebras. We will explore this direction in a future work.

The paper is organized as follows. In Section \ref{sec:L}, we introduce the notion of a twilled 3-Lie algebra, and construct an $L_\infty$-algebra using Uchino's higher derived brackets. In Section \ref{sec:GM}, we introduce the notion of a generalized matched pair of 3-Lie algebras and show that a generalized matched pair gives rise to a twilled 3-Lie algebra. We further introduce the notion of a strict twilled 3-Lie algebra and show that it is equivalent to a (usual) matched pair of 3-Lie algebras. In Section \ref{sec:T}, we study twisting of twilled 3-Lie algebras, and show that the twisting of a twilled 3-Lie algebra by a Maurer-Cartan element of the $L_\infty$-algebra given in Section \ref{sec:L} is still a twilled 3-Lie algebra. Finally in Section \ref{sec:OT}, we construct twilled 3-Lie algebras using $\huaO$-operators on 3-Lie algebras and give examples.

\vspace{2mm}

In this paper, we work over an algebraically closed field $\K$ of characteristic 0 and all the vector spaces are over $\K$ and finite-dimensional.

\vspace{2mm}
\noindent
{\bf Acknowledgements. } Y. Sheng is  supported by NSFC
(11922110). R. Tang is supported by NSFC
(12001228) and China Postdoctoral Science
Foundation (2020M670833).

\section{Twilled $3$-Lie algebras and the associated $L_\infty$-algebras}\label{sec:L}

In this section, we introduce the notion of a twilled 3-Lie algebra, from which we construct an $L_\infty$-algebra using Uchino's higher derived brackets. First we recall 3-Lie algebras and the graded Lie algebra that characterize 3-Lie algebras as Maurer-Cartan elements.

\begin{defi}{\rm (\cite{Filippov})}
A {\bf 3-Lie algebra} is a vector space $\g$ together with a skew-symmetric linear map $[\cdot,\cdot,\cdot]_{\g}:\otimes^{3}\g\rightarrow \g$, such that for $ x_{i}\in \g, 1\leq i\leq 5$, the following {\bf Fundamental Identity} holds:
\begin{eqnarray*}
\qquad [x_1,x_2,[x_3,x_4, x_5]_{\g}]_{\g}=[[x_1,x_2, x_3]_{\g},x_4,x_5]_{\g}+[x_3,[x_1,x_2, x_4]_{\g},x_5]_{\g}+[x_3,x_4,[x_1,x_2, x_5]_{\g}]_{\g}.
 \label{eq:jacobi1}
\end{eqnarray*}
\end{defi}

A permutation $\sigma\in \huaS_n$ is called an $(i,n-i)$-{\bf shuffle} if $\sigma(1)<\cdots <\sigma(i)$ and $\sigma(i+1)<\cdots <\sigma(n)$. If $i=0$ or $n$ we assume $\sigma=\Id$. The set of all $(i,n-i)$-shuffles will be denoted by $\mathbb \huaS_{(i,n-i)}$. The notion of an $(i_1,\cdots,i_k)$-shuffle and the set $\huaS_{(i_1,\cdots,i_k)}$ are defined analogously.

Let $\g$ be a vector space. We consider the graded vector space $$C^*_{\Li}(\g,\g)=\oplus_{n\ge 0}C^n_{\Li}(\g,\g)=\oplus_{n\ge 0}\Hom (\underbrace{\wedge^{2} \g\otimes \cdots\otimes \wedge^{2}\g}_{n}\wedge \g, \g),~(n\geq 0).$$ It is known that $C^*_{\Li}(\g,\g)$ equipped with the {\bf grade commutator bracket}
\begin{eqnarray}\label{3-Lie-bracket}
[P,Q]_{\Li}=P{\circ}Q-(-1)^{pq}Q{\circ}P,\quad \forall~ P\in C^{p}_{\Li}(\g,\g),Q\in C^{q}_{\Li}(\g,\g),
\end{eqnarray}
is a graded Lie algebra \cite{NR bracket of n-Lie}, where $P{\circ}Q\in C_{\Li}^{p+q}(\g,\g)$ is defined by

{\footnotesize
\begin{equation*}
\begin{aligned}
&(P{\circ}Q)(\mathfrak{X}_1,\cdots,\mathfrak{X}_{p+q},x)\\
=&\sum_{k=1}^{p}(-1)^{(k-1)q}\sum_{\sigma\in \huaS(k-1,q)}(-1)^\sigma P(\mathfrak{X}_{\sigma(1)},\cdots,\mathfrak{X}_{\sigma(k-1)},
Q(\mathfrak{X}_{\sigma(k)},\cdots,\mathfrak{X}_{\sigma(k+q-1)},x_{k+q})\wedge y_{k+q},\mathfrak{X}_{k+q+1},\cdots,\mathfrak{X}_{p+q},x)\\
&+\sum_{k=1}^{p}(-1)^{(k-1)q}\sum_{\sigma\in \huaS(k-1,q)}(-1)^\sigma P(\mathfrak{X}_{\sigma(1)},\cdots,\mathfrak{X}_{\sigma(k-1)},x_{k+q}\wedge
Q(\mathfrak{X}_{\sigma(k)},\cdots,\mathfrak{X}_{\sigma(k+q-1)},y_{k+q}),\mathfrak{X}_{k+q+1},\cdots,\mathfrak{X}_{p+q},x)\\
&+\sum_{\sigma\in \huaS(p,q)}(-1)^{pq}(-1)^\sigma P(\mathfrak{X}_{\sigma(1)},\cdots,\mathfrak{X}_{\sigma(p)},
Q(\mathfrak{X}_{\sigma(p+1)},\cdots,\mathfrak{X}_{\sigma(p+q-1)},\mathfrak{X}_{\sigma(p+q)},x)),\\
\end{aligned}
\end{equation*}
}
where   $\mathfrak{X}_{i}=x_i\wedge y_i\in \wedge^2 \g$, $i=1,2,\cdots,p+q$ and $x\in\g.$

The following result is well known.
\begin{pro}\label{pro:3LieMC}
  Let  $\pi \in C^{1}_{\Li}(\g,\g)=\Hom(\wedge^3\g,\g)$. Then $\pi $ defines a $3$-Lie algebra structure on $\g$ if and only if $\pi$ satisfies the Maurer-Cartan equation $[\pi ,\pi ]_{\Li}=0$.

Moreover, $(C^*_{\Li}(\g,\g),[\cdot,\cdot]_{\Li},d_{\pi})$ is a differential graded Lie algebra, where $d_\pi$ is defined by
\begin{eqnarray}
d_\pi:=[\pi,\cdot]_{\Li}.
\end{eqnarray}
\end{pro}

The differential $d_\pi$ can be served as the coboundary operator of 3-Lie algebras. See \cite{Takhtajan1} for the cohomology theory for $n$-Lie algebras.

\subsection{Lift and bidegree}\hspace{2mm}

Let $\g_1$ and $\g_2$ be vector spaces. We denote by $\g^{l,k}$ the subspace of $\wedge^{2} (\g_1\oplus\g_2)\otimes$$\overset{(n)}{\cdots}$$ \otimes \wedge^{2} (\g_1\oplus\g_2)\wedge (\g_1\oplus\g_2)$
which contains the number of $\g_1$ (resp. $\g_2$) is $l$ (resp. $k$).
Then the vector space $\wedge^{2} (\g_1\oplus\g_2)\otimes$$\overset{(n)}{\cdots}$$ \otimes \wedge^{2} (\g_1\oplus\g_2)\wedge (\g_1\oplus\g_2)$ is expended into the direct sum of $\g^{l,k},~l+k=2n+1$. For instance,
$$
\wedge^3(\g_1\oplus\g_2)=\g^{3,0}\oplus\g^{2,1}\oplus\g^{1,2}\oplus\g^{0,3}.
$$
It is straightforward to see that we have the following natural isomorphism
\begin{eqnarray}\label{decomposition}
C^n_{\Li}(\g_1\oplus\g_2,\g_1\oplus\g_2)\cong\sum_{l+k=2n+1}\Hom(\g^{l,k},\g_1)\oplus\sum_{l+k=2n+1}\Hom(\g^{l,k},\g_2).
\end{eqnarray}
An element $f\in\Hom(\g^{l,k},\g_1)$ (resp. $f\in\Hom(\g^{l,k},\g_2)$) naturally gives an element $\hat{f}\in C^n_{\Li}(\g_1\oplus\g_2,\g_1\oplus\g_2)$, which is called its  lift.
For example, the lifts of linear maps $\alpha:\wedge^3\g_1\lon\g_1,~\beta:\wedge^2\g_1\otimes\g_2\lon\g_2$ are defined by
\begin{eqnarray}
\label{semidirect-1}\hat{\alpha}\big((x,u),(y,v),(z,w)\big)&=&(\alpha(x,y,z),0),\\
\label{semidirect-2}\hat{\beta}\big((x,u),(y,v),(z,w)\big)&=&(0,\beta(x,y,w)+\beta(y,z,u)+\beta(z,x,v)),
\end{eqnarray}
respectively. Let $H:\g_2\lon\g_1$ be a linear map. Its lift is given by
$
\hat{H}(x,u)=(H(u),0).
$
Obviously we have $\hat{H}\circ\hat{H}=0.$

\begin{defi}\label{Bidegree}
A linear map $f\in \Hom(\underbrace{\wedge^{2} (\g_1\oplus\g_2)\otimes{\cdots} \otimes \wedge^{2} (\g_1\oplus\g_2)}_n\wedge (\g_1\oplus\g_2),\g_1\oplus\g_2)$ has a {\bf bidegree} $l|k$, which is denoted by $||f||=l|k$,   if $f$ satisfies the following four conditions:
\begin{itemize}
\item[\rm(i)] $l+k=2n;$
\item[\rm(ii)] If $X$ is an element in $\g^{l+1,k}$, then $f(X)\in\g_1;$
\item[\rm(iii)] If $X$ is an element in $\g^{l,k+1}$, then $f(X)\in\g_2;$
\item[\rm(iv)] All the other case, $f(X)=0.$
\end{itemize}
\end{defi}
A linear map $f$ is said to be homogeneous  if $f$ has a bidegree.
We have $l+k\ge0,~k,l\ge-1$ because $n\ge0$ and $l+1,~k+1\ge0$. For instance, the lift $\hat{H}\in C^0_{\Li}(\g_1\oplus\g_2,\g_1\oplus\g_2)$ of $H:\g_2\lon\g_1$ has the bidegree $-1|1$. Linear maps $\hat{\alpha},~\hat{\beta}\in C^1_{\Li}(\g_1\oplus\g_2,\g_1\oplus\g_2)$ given in \eqref{semidirect-1} and \eqref{semidirect-2} have the bidegree  $||\hat{\alpha}||=||\hat{\beta}||=2|0$. Thus, the sum
\begin{eqnarray}
\label{semidirect}\hat{\mu}:=\hat{\alpha}+\hat{\beta}
\end{eqnarray}
is a homogeneous linear map of the bidegree $2|0$, which is also a multiplication of the semidirect product type,
$$
\hat{\mu}\big((x,u),(y,v),(z,w)\big)=(\alpha(x,y,z),\beta(x,y,w)+\beta(y,z,u)+\beta(z,x,v)).
$$

It is obvious that we have the following lemmas:
\emptycomment{
\begin{lem}
  The bidegree of $f\in C^n(\g_1\oplus\g_2,\g_1\oplus\g_2)$ is $l|k$ if and only if the following four conditions hold:
\begin{itemize}
\item[\rm(i)] $l+k+1=n;$
\item[\rm(ii)] If $X$ is an element in $\g^{l+1,k}$, then $f(X)\in\g_1;$
\item[\rm(iii)] If $X$ is an element in $\g^{l,k+1}$, then $f(X)\in\g_2;$
\item[\rm(iv)] All the other case, $f(X)=0.$
\end{itemize}
\end{lem}
}

\begin{lem}\label{Zero-condition-1}
Let $f_1,\cdots,f_k\in C^n_{\Li}(\g_1\oplus\g_2,\g_1\oplus\g_2)$ be homogeneous linear maps and the bidegrees of $f_i$ are
different. Then $f_1+\cdots+f_k=0$ if and only if $f_1=\cdots=f_k=0.$
\end{lem}

\begin{lem}\label{Zero-condition-2}
If $||f||=-1|l$ (resp. $l|-1$) and $||g||=-1|k$ (resp. $k|-1$), then $[f,g]_{\Li}=0.$
\end{lem}
\begin{proof}
Assume that $||f||=-1|l$ and $||g||=-1|k$. Then $f$ and $g$ are both lifts of linear maps in $\oplus_{m\geq0}\Hom (\underbrace{\wedge^{2} \g_2\otimes \cdots\otimes \wedge^{2}\g_2}_{m}\wedge \g_2, \g_1)$. Thus, we have $f\circ g=g\circ f=0$ and $[f,g]_{\Li}=0.$
  \end{proof}

\begin{lem}\label{important-lemma-1}
Let $f\in C^{n}_{\Li}(\g_1\oplus\g_2,\g_1\oplus\g_2)$ and $g\in C^{m}_{\Li}(\g_1\oplus\g_2,\g_1\oplus\g_2)$ be homogeneous linear maps with bidegrees $l_f|k_f$ and $l_g|k_g$ respectively. Then the composition $f\circ g\in C^{n+m}_{\Li}(\g_1\oplus\g_2,\g_1\oplus\g_2)$ is  a homogeneous linear map of the bidegree $(l_f+l_g)|(k_f+k_g).$
\end{lem}
\begin{proof}
  It is straightforward.
\end{proof}

\begin{lem}\label{important-lemma-2}
If $||f||=l_f|k_f$ and $||g||=l_g|k_g$, then $[f,g]_{\Li}$ has the bidegree $(l_f+l_g)|(k_f+k_g).$
\end{lem}
\begin{proof}
By Lemma \ref{important-lemma-1} and \eqref{3-Lie-bracket}, we have $||[f,g]_{\Li}||=(l_f+l_g)|(k_f+k_g).$
\end{proof}

\subsection{Twilled $3$-Lie algebras}

Let $(\huaG,[\cdot,\cdot,\cdot]_\huaG)$ be a $3$-Lie algebra with a decomposition into two subspaces\footnote{Here $\g_1$ and $\g_2$ are not necessarily subalgebras.}, $\huaG=\g_1\oplus\g_2$.

\begin{lem}\label{lem:dec}
Any $\Omega\in C^1_{\Li}(\huaG,\huaG)=\Hom(\wedge^3(\g_1\oplus \g_2),\g_1\oplus \g_2)$ is uniquely decomposed into five homogeneous linear maps of bidegrees $3|-1,~2|0,~1|1,~0|2$ and~$-1|3:$
$$
\Omega=\hat{\phi}_1+\hat{\mu}_1+\hat{\psi}+\hat{\mu}_2+\hat{\phi}_2.
$$
\end{lem}
\begin{proof}
By \eqref{decomposition}, $C^1_{\Li}(\huaG,\huaG)$ is decomposed into
$$
C^1_{\Li}(\huaG,\huaG)=(3|-1)\oplus(2|0)\oplus(1|1)\oplus(0|2)\oplus(-1|3),
$$
where $(i|j)$ is the space of linear maps  of the bidegree $i|j$. By Lemma \ref{Zero-condition-1}, $\Omega$ is uniquely decomposed into homogeneous linear maps of bidegrees~$3|-1,~2|0,~1|1,~0|2$~and~$-1|3$.
\end{proof}

The multiplication $[(x,u),(y,v),(z,w)]_{\huaG}$ of $\huaG$ is uniquely decomposed by the canonical projections $\huaG\lon\g_1$ and $\huaG\lon\g_2$ into eight multiplications:
\begin{eqnarray*}
 \begin{array}{rclrcl}~[x,y,z]_{\huaG}&=&([x,y,z]_1,[x,y,z]_2),& [x,y,w]_{\huaG}&=&([x,y,w]_1,[x,y,w]_2),\\
~ [x,v,w]_{\huaG}&=&([x,v,w]_1,[x,v,w]_2),&  [u,v,w]_{\huaG}&=&([u,v,w]_1,[u,v,w]_2),
 \end{array}
\end{eqnarray*}
where $[\alpha,\beta,\gamma]_1$ (resp. $[\alpha,\beta,\gamma]_2$) means the projection of $[\alpha,\beta,\gamma]_\huaG$ to $\g_1$ (resp. $\g_2$) for all $\alpha,\beta,\gamma\in\huaG$.  For later convenience, we use $\Omega$ to denote the multiplication $[\cdot,\cdot,\cdot]_\huaG$, i.e. $$\Omega((x,u),(y,v),(z,w)):=[(x,u),(y,v),(z,w)]_{\huaG}.$$
Write $\Omega=\hat{\phi}_1+\hat{\mu}_1+\hat{\psi}+\hat{\mu}_2+\hat{\phi}_2$ as in Lemma \ref{lem:dec}.  Then we obtain
\begin{small}
\begin{eqnarray}
\label{bracket-1}\hat{\phi}_1((x,u),(y,v),(z,w))&=&(0,[x,y,z]_2),\\
\label{bracket-2}\hat{\mu}_1((x,u),(y,v),(z,w))&=&([x,y,z]_1,[x,y,w]_2+[u,y,z]_2+[x,v,z]_2),\\
\label{bracket-3}\hat{\psi}~((x,u),(y,v),(z,w))&=&([x,y,w]_1+[u,y,z]_1+[x,v,z]_1,[x,v,w]_2+[u,y,w]_2+[u,v,z]_2),\\
\label{bracket-4}\hat{\mu}_2((x,u),(y,v),(z,w))&=&([x,v,w]_1+[u,y,w]_1+[u,v,z]_1,[u,v,w]_2),\\
\label{bracket-5}\hat{\phi}_2~((x,u),(y,v,(z,w))&=&([u,v,w]_1,0).
\end{eqnarray}
\end{small}
Observe that $\hat{\phi}_1$ and $\hat{\phi}_2$ are lifted linear maps of $\phi_1(x,y,z):=[x,y,z]_2$ and $\phi_2(u,v,w):=[u,v,w]_1.$
\begin{lem}\label{proto-twilled}
The Maurer-Cartan equation $[\Omega,\Omega]_{\Li}=0$ is equivalent to the following  conditions:
\begin{eqnarray}\label{eq:OTTT}
\left\{\begin{array}{rcl}
{}[\hat{\phi}_1,\hat{\mu}_1]_{\Li}&=&0,\\
{}[\hat{\phi}_1,\hat{\psi}]_{\Li}+\frac{1}{2}[\hat{\mu}_1,\hat{\mu}_1]_{\Li}&=&0,\\
{}[\hat{\phi}_1,\hat{\mu}_2]_{\Li}+[\hat{\psi},\hat{\mu}_1]_{\Li}&=&0,\\
{}[\hat{\mu}_1,\hat{\phi}_2]_{\Li}+[\hat{\psi},\hat{\mu}_2]_{\Li}&=&0,\\
{}[\hat{\phi}_1,\hat{\phi}_2]_{\Li}+[\hat{\mu}_1,\hat{\mu}_2]_{\Li}+\frac{1}{2}[\hat{\psi},\hat{\psi}]_{\Li}&=&0,\\
{}[\hat{\psi},\hat{\phi}_2]_{\Li}+\frac{1}{2}[\hat{\mu}_2,\hat{\mu}_2]_{\Li}&=&0,\\
{}[\hat{\mu}_2,\hat{\phi}_2]_{\Li}&=&0.
\end{array}\right.
\end{eqnarray}
\end{lem}
\begin{proof}
By Lemma \ref{Zero-condition-2}, we have
\begin{eqnarray*}
[\Omega,\Omega]_{\Li}&=&[\hat{\phi}_1+\hat{\mu}_1+\hat{\psi}+\hat{\mu}_2+\hat{\phi}_2,\hat{\phi}_1+\hat{\mu}_1+\hat{\psi}+\hat{\mu}_2+\hat{\phi}_2]_{\Li}\\
               &=&[\hat{\phi}_1,\hat{\mu}_1]_{\Li}+[\hat{\phi}_1,\hat{\psi}]_{\Li}+[\hat{\phi}_1,\hat{\mu}_2]_{\Li}+[\hat{\phi}_1,\hat{\phi}_2]_{\Li}+[\hat{\mu}_1,\hat{\phi}_1]_{\Li}+[\hat{\mu}_1,\hat{\mu}_1]_{\Li}\\
               &&+[\hat{\mu}_1,\hat{\psi}]_{\Li}+[\hat{\mu}_1,\hat{\mu}_2]_{\Li}+[\hat{\mu}_1,\hat{\phi}_2]_{\Li}+[\hat{\psi}~,\hat{\phi}_1]_{\Li}+[\hat{\psi}~,\hat{\mu}_1]_{\Li}+[\hat{\psi}~,\hat{\psi}]_{\Li}\\
               &&+[\hat{\psi}~,\hat{\mu}_2]_{\Li}+[\hat{\psi}~,\hat{\phi}_2]_{\Li}+[\hat{\mu}_2,\hat{\phi}_1]_{\Li}+[\hat{\mu}_2,\hat{\mu}_1]_{\Li}+[\hat{\mu}_2,\hat{\psi}]_{\Li}+[\hat{\mu}_2,\hat{\mu}_2]_{\Li}\\
               &&+[\hat{\mu}_2,\hat{\phi}_2]_{\Li}+[\hat{\phi}_2,\hat{\phi}_1]_{\Li}+[\hat{\phi}_2,\hat{\mu}_1]_{\Li}+[\hat{\phi}_2,\hat{\psi}]_{\Li}+[\hat{\phi}_2,\hat{\mu}_2]_{\Li}\\
               &=&2[\hat{\phi}_1,\hat{\mu}_1]_{\Li}+(2[\hat{\phi}_1,\hat{\psi}]_{\Li}+[\hat{\mu}_1,\hat{\mu}_1]_{\Li})+(2[\hat{\phi}_1,\hat{\mu}_2]_{\Li}+2[\hat{\psi},\hat{\mu}_1]_{\Li})\\
               &&+(2[\hat{\mu}_1,\hat{\phi}_2]_{\Li}+2[\hat{\psi},\hat{\mu}_2]_{\Li})+(2[\hat{\phi}_1,\hat{\phi}_2]_{\Li}+2[\hat{\mu}_1,\hat{\mu}_2]_{\Li}+[\hat{\psi},\hat{\psi}]_{\Li})\\
               &&+(2[\hat{\psi},\hat{\phi}_2]_{\Li}+[\hat{\mu}_2,\hat{\mu}_2]_{\Li})+2[\hat{\mu}_2,\hat{\phi}_2]_{\Li}.
\end{eqnarray*}
By Lemma \ref{important-lemma-2} and Lemma \ref{Zero-condition-1},  $[\Omega,\Omega]_{\Li}=0$ if and only if  \eqref{eq:OTTT} holds.
\end{proof}

\begin{defi}
The triple $(\huaG,\g_1,\g_2)$ is called a {\bf twilled $3$-Lie algebra} if $\phi_1=\phi_2=0$, or equivalently, $\g_1$ and $\g_2$ are subalgebras of $\huaG$.
 \end{defi}

 \begin{rmk}
   Recall from \cite{Sheng-Tang} that a product structure on a $3$-Lie algebra $(\huaG,[\cdot,\cdot,\cdot]_\huaG)$ is a linear map $E:\huaG\lon\huaG$ such that $E^2=\Id$ and \begin{eqnarray*}\nonumber\label{product-structure}
E[x,y,z]_\huaG&=&[Ex,Ey,Ez]_\huaG+[Ex,y,z]_\huaG+[x,Ey,z]_\huaG+[x,y,Ez]_\huaG\\
&&-E[Ex,Ey,z]_\huaG-E[x,Ey,Ez]_\huaG-E[Ex,y,Ez]_\huaG.
\end{eqnarray*}
A $3$-Lie algebra admits a product structure if and only if as vector space $\huaG=\huaG_1\oplus \huaG_2$, where $\huaG_1$ and $ \huaG_2$ are subalgebras. See \cite[Theorem 5.3]{Sheng-Tang} for more details. Therefore, there exists a product structure on a twilled $3$-Lie algebra.
 \end{rmk}

 By Lemma \ref{proto-twilled}, we have the following corollary.

\begin{cor}\label{lem:twilled-t}
The triple $(\huaG,\g_1,\g_2)$ is a twilled $3$-Lie algebra if and only if the following five conditions hold:
\begin{eqnarray}
\label{twilled-1}\frac{1}{2}[\hat{\mu}_1,\hat{\mu}_1]_{\Li}&=&0,\\
\label{twilled-2}[\hat{\psi},\hat{\mu}_1]_{\Li}&=&0,\\
\label{twilled-3}[\hat{\psi},\hat{\mu}_2]_{\Li}&=&0,\\
\label{twilled-4}[\hat{\mu}_1,\hat{\mu}_2]_{\Li}+\frac{1}{2}[\hat{\psi},\hat{\psi}]_{\Li}&=&0,\\
\label{twilled-5}\frac{1}{2}[\hat{\mu}_2,\hat{\mu}_2]_{\Li}&=&0.
\end{eqnarray}
\end{cor}

\subsection{The associated $L_\infty$-algebras}
The notion of an $L_\infty$-algebra was introduced by Schlessinger and Stasheff in \cite{SS85,stasheff:shla}. See  \cite{LS,LM} for more details.
\begin{defi}
An {\bf  $L_\infty$-algebra} is a $\mathbb Z$-graded vector space $\g=\oplus_{k\in\mathbb Z}\g^k$ equipped with a collection $(k\ge 1)$ of linear maps $l_k:\otimes^k\g\lon\g$ of degree $1$ with the property that, for any homogeneous elements $x_1,\cdots,x_n\in \g$, we have
\begin{itemize}\item[\rm(i)]
{\em (graded symmetry)} for every $\sigma\in\huaS_{n}$,
\begin{eqnarray*}
l_n(x_{\sigma(1)},\cdots,x_{\sigma(n-1)},x_{\sigma(n)})=\varepsilon(\sigma)l_n(x_1,\cdots,x_{n-1},x_n),
\end{eqnarray*}
\item[\rm(ii)] {\em (generalized Jacobi identity)} for all $n\ge 1$,
\begin{eqnarray*}\label{sh-Lie}
\sum_{i=1}^{n}\sum_{\sigma\in \mathbb S_{(i,n-i)} }\varepsilon(\sigma)l_{n-i+1}(l_i(x_{\sigma(1)},\cdots,x_{\sigma(i)}),x_{\sigma(i+1)},\cdots,x_{\sigma(n)})=0.
\end{eqnarray*}
\end{itemize}
\end{defi}

\begin{defi}
An element $\alpha\in \g^0$ is called a {\bf Maurer-Cartan element} of an $L_\infty$-algebra $(\g,\{l_k\}_{k=1}^{+\infty})$ if it satisfies the Maurer-Cartan equation
\begin{eqnarray}\label{MC-equation}
\sum_{k=1}^{+\infty}\frac{1}{k!}l_k(\alpha,\cdots,\alpha)=0.
\end{eqnarray}
\end{defi}

One method for constructing explicit $L_\infty$-algebras is given by Uchino's higher derived brackets \cite{Uchino-1,Uchino-2}. Let us recall this construction.

\begin{defi}{\rm (\cite{Uchino-1,Uchino-2})}
An {\bf$\U$-structure} consists of a pair $(L,\{{d_i}\}_{i=1}^{+\infty})$ where
\begin{itemize}
\item[$\bullet$] $(L,[\cdot,\cdot])$ is a graded Lie algebra,
\item[$\bullet$] for any $i$,~ $d_i$ is a derivation of $(L,[\cdot,\cdot])$  and satisfies the following equation
$$
\sum_{i+j=n\atop i,j\ge0}d_i\circ d_j=0,\,\,\,\,n\in\mathbb Z_{\ge0}.
$$.
\end{itemize}
\end{defi}

\begin{thm}{\rm (\cite{Uchino-1,Uchino-2})}\label{thm:db}
Let $(L,\{{d_i}\}_{i=1}^{+\infty})$ be an $\U$-structure. Then $(L,\{{l_k}\}_{k=1}^{+\infty})$ is a Leibniz$_\infty$-algebra where
\begin{eqnarray}\label{V-shla}
l_k(x_1,\cdots,x_k)=\underbrace{[\cdots[[}_k d_{k-1}(x_1),x_2],x_3],\cdots,x_k],\quad\mbox{for homogeneous}~   x_1,\cdots,x_k\in L.
\end{eqnarray}
We call $\{{l_k}\}_{k=1}^{+\infty}$ the {\bf higher derived brackets} of the $\U$-structure $(L,\{{d_i}\}_{i=1}^{+\infty})$.

Let $\h$ be an abelian graded Lie subalgebra of $(L,[\cdot,\cdot])$. If $l_k$ is closed on $\h$, then $(\h,\{l_k\}_{k=1}^{+\infty})$ is an $L_\infty$-algebra.
\end{thm}

Let $(\huaG,\g_1,\g_2)$ be a twilled $3$-Lie algebra. Denote by $$C^m_{\Li}(\g_2,\g_1)=\Hom (\underbrace{\wedge^{2} \g_2\otimes \cdots\otimes \wedge^{2}\g_2}_{m}\wedge \g_2, \g_1),~(m\geq 0),$$  and  define
\begin{eqnarray*}
&&l_1:C^m_{\Li}(\g_2,\g_1)\lon C^{m+1}_{\Li}(\g_2,\g_1),\\
&&l_2:C^m_{\Li}(\g_2,\g_1)\times C^n_{\Li}(\g_2,\g_1)\lon C^{m+n+1}_{\Li}(\g_2,\g_1),\\
&&l_3:C^m_{\Li}(\g_2,\g_1)\times C^n_{\Li}(\g_2,\g_1)\times C^r_{\Li}(\g_2,\g_1)\lon C^{m+n+r+1}_{\Li}(\g_2,\g_1)
\end{eqnarray*} by
\begin{eqnarray*}
l_1(g_1)&=&[\hat{\mu}_2,\hat{g}_1]_{\Li},\\
l_2(g_1,g_2)&=&[[\hat{\psi},\hat{g}_1]_{\Li},\hat{g}_2]_{\Li},\\
l_3(g_1,g_2,g_3)&=&[[[\hat{\mu}_1,\hat{g}_1]_{\Li},\hat{g}_2]_{\Li},\hat{g}_3]_{\Li},
\end{eqnarray*}
for all $g_1\in C^m_{\Li}(\g_2,\g_1),~g_2\in C^n_{\Li}(\g_2,\g_1),~g_3\in C^r_{\Li}(\g_2,\g_1).$

\begin{thm}\label{quasi-as-shLie}
 Let $(\huaG,\g_1,\g_2)$ be a twilled $3$-Lie algebra. Then $(\oplus_{m\geq0}C^m_{\Li}(\g_2,\g_1),l_1,l_2,l_3)$ is an $L_\infty$-algebra.
\end{thm}
\begin{proof}
  We set $d_0:=[\hat{\mu}_2,\cdot]_{\Li}$. By \eqref{twilled-5} and the fact that $(C^*_{\Li}(\huaG,\huaG),[\cdot,\cdot]_{\Li})$ is a graded Lie algebra, we deduce that $(C^*_{\Li}(\huaG,\huaG),[\cdot,\cdot]_{\Li},d_0)$ is a differential graded Lie algebra. Moreover, we define
$$
d_1:=[\hat{\psi},\cdot]_{\Li},\quad d_2:=[\hat{\mu}_1,\cdot]_{\Li},\quad d_i=0,\quad\forall i\ge3.
$$
By \eqref{twilled-3}, for all $f\in C^*_{\Li}(\huaG,\huaG)$, we have
\begin{eqnarray*}
(d_0\circ d_1+d_1\circ d_0)(f)=[\hat{\mu}_2,[\hat{\psi},f]_{\Li}]_{\Li}+[\hat{\psi},[\hat{\mu}_2,f]_{\Li}]_{\Li}=[[\hat{\mu}_2,\hat{\psi}]_{\Li},f]_{\Li}=0,
\end{eqnarray*}
which implies that $d_0\circ d_1+d_1\circ d_0=0.$ By \eqref{twilled-4},   we have
\begin{eqnarray*}
(d_0\circ d_2+d_1\circ d_1+d_2\circ d_0)(f)
                              &=&[\hat{\mu}_2,[\hat{\mu}_1,f]_{\Li}]_{\Li}+[\hat{\psi},[\hat{\psi},f]_{\Li}]_{\Li}
                              +[\hat{\mu}_1,[\hat{\mu}_2,f]_{\Li}]_{\Li}\\
                              &=&[[\hat{\mu}_1,\hat{\mu}_2]_{\Li},f]_{\Li}+\frac{1}{2}[[\hat{\psi},\hat{\psi}]_{\Li},f]_{\Li}\\
                              &=&0.
\end{eqnarray*}
Thus, we have $d_0\circ d_2+d_1\circ d_1+d_2\circ d_0=0.$ By \eqref{twilled-2},   we have
\begin{eqnarray*}
(d_1\circ d_2+d_2\circ d_1)(f)=[\hat{\psi},[\hat{\mu}_1,f]_{\Li}]_{\Li}+[\hat{\mu}_1,[\hat{\psi},f]_{\Li}]_{\Li}=[[\hat{\psi},\hat{\mu}_1]_{\Li},f]_{\Li}=0.
\end{eqnarray*}
Thus, we have $d_1\circ d_2+d_2\circ d_1=0.$ By \eqref{twilled-1}, we have
\begin{eqnarray*}
(d_2\circ d_2)(f)=d_2(d_2(f))=[\hat{\mu}_1,[\hat{\mu}_1,f]_{\Li}]_{\Li}=\frac{1}{2}[[\hat{\mu}_1,\hat{\mu}_1]_{\Li},f]_{\Li}=0.
\end{eqnarray*}
For $i\ge3$, $d_i=0$. Thus, we obtain
$
\sum_{i+j=n\atop i,j\ge0}d_i\circ d_j=0,\,\,n\in\mathbb Z_{\ge0}.
$
Therefore, we have higher derived brackets on $C^*_{\Li}(\huaG,\huaG)$ given by
\begin{eqnarray*}
l_1(f_1)&=&[\hat{\mu}_2,f_1]_{\Li},\\
l_2(f_1,f_2)&=&[[\hat{\psi},f_1]_{\Li},f_2]_{\Li},\\
l_3(f_1,f_2,f_3)&=&[[[\hat{\mu}_1,f_1]_{\Li},f_2]_{\Li},f_3]_{\Li},\\
l_i&=&0,\,\,\,\,i\ge4,
\end{eqnarray*}
where $f_1,~f_2,~f_3\in C^*_{\Li}(\huaG,\huaG)$.

 Moreover, it is obvious that $\oplus_{m\geq0}C^m_{\Li}(\g_2,\g_1)$ is an abelian subalgebra of the graded Lie algebra $(C^*_{\Li}(\huaG,\huaG),[\cdot,\cdot]_{\Li})$. Next we show that $l_1,~l_2,~l_3$ are closed on $\oplus_{m\geq0}C^m_{\Li}(\g_2,\g_1)$. For all $g_1\in C^m_{\Li}(\g_2,\g_1)$, we have $||\hat{g}_1||=-1|2m+1$. By Lemma \ref{important-lemma-2}, we deduce that $||l_1(\hat{g}_1)||=||[\hat{\mu}_2,\hat{g}_1]_{\Li}||=-1|2(m+1)+1$. Thus, we have $l_1(\hat{g}_1)\in \oplus_{m\geq0}C^m_{\Li}(\g_2,\g_1)$. Similarly, we have
\begin{eqnarray*}
||l_2(\hat{g}_1,\hat{g}_2)||=-1|2(m+n+1)+1,\quad
||l_3(\hat{g}_1,\hat{g}_2,\hat{g}_3)||=-1|2(m+n+r+1)+1.
\end{eqnarray*}
Thus, $l_1,~l_2,~l_3$ are closed on $\oplus_{m\geq0}C^m_{\Li}(\g_2,\g_1)$. By Theorem \ref{thm:db}, $(\oplus_{m\geq0}C^m_{\Li}(\g_2,\g_1),l_1,l_2,l_3)$ is an $L_\infty$-algebra. The proof is finished.
\end{proof}

\section{Generalized matched pairs of $3$-Lie algebras}\label{sec:GM}

In this section, first we recall representations and matched pairs of $3$-Lie algebras. Then we introduce the notion of a generalized matched pair of $3$-Lie algebras using the generalized representation introduced in \cite{Jiefeng}. We show that a generalized matched pair gives rise to a twilled $3$-Lie algebra. Finally we introduce the notion of a strict twilled $3$-Lie algebra and show that there is a one-to-one correspondence between strict twilled $3$-Lie algebras and matched pairs of $3$-Lie algebras.

\begin{defi}{\rm (\cite{Dzhu, KA})}
A {\bf representation}  of a $3$-Lie algebra $(\g,[\cdot,\cdot,\cdot]_{\g})$ on a vector space $V$ is a linear
map: $\rho:\wedge^{2}\g\rightarrow \gl(V)$, such that for all $x_{1}, x_{2}, x_{3}, x_{4}\in \g,$ there holds:
\begin{eqnarray}
~\label{representation-1}\rho(x_{1},x_{2})\rho(x_{3},x_{4})&=&\rho([x_{1},x_{2},x_{3}]_{\g},x_{4})+
\rho(x_{3},[x_{1},x_{2},x_{4}]_{\g})+\rho(x_{3},x_{4})\rho(x_{1},x_{2});\\
~\label{representation-2}\rho(x_{1},[x_{2},x_{3},x_{4}]_{\g})&=&\rho(x_{3},x_{4})\rho(x_{1},x_{2})-\rho(x_{2},x_{4})\rho(x_{1},x_{3})
+\rho(x_{2},x_{3})\rho(x_{1},x_{4}).
\end{eqnarray}
\end{defi}

\begin{ex}
Let $(\g,[\cdot,\cdot,\cdot]_{\g})$ be a $3$-Lie algebra.  Define $\ad:\wedge^{2}\g\rightarrow \gl(\g)$ by
\begin{eqnarray}\label{eq2}
\ad_{x,y}z:=[x,y,z]_{\g},\quad \forall x,y,z\in \g.
\end{eqnarray}
Then   $(\g;\ad)$ is a representation of $(\g,[\cdot,\cdot,\cdot]_{\g})$, which is called the {\bf adjoint representation} of $\g$.
\end{ex}

\begin{defi}{\rm (\cite{BGS-3-Bialgebras})}
Let $(\g_1,[\cdot,\cdot,\cdot]_{\g_1})$ and $(\g_2,[\cdot,\cdot,\cdot]_{\g_2})$ be two $3$-Lie algebras. Suppose that there are linear maps $\rho:\wedge^2\g_1\rightarrow\gl(\g_2)$ and $\mu:\wedge^2\g_2\rightarrow\gl(\g_1)$ such that $(\g_2;\rho)$ is a representation of $\g_1$ and
 $(\g_1;\mu)$ is a representation of $\g_2$. For all $x,y,z\in \g_1$ and $u,v,w\in \g_2,$ if $\rho$ and $\mu$ satisfy the following conditions:
\begin{eqnarray}
~\mu(u,v)[x,y,z]_{\g_1}&=&[\mu(u,v)x,y,z]_{\g_1}+[x,\mu(u,v)y,z]_{\g_1}+[x,y,\mu(u,v)z]_{\g_1};\label{eq:mp1}\\
~[x,y,\mu(u,v)z]_{\g_1}&=&\mu(\rho(x,y)u,v)z-\mu (\rho(x,z)v,u)y+\mu (\rho(y,z)v,u)x;\label{eq:mp2}\\
~[\mu(u,v)x,y,z]_{\g_1}&=&\mu(u,v)[x,y,z]_{\g_1}+\mu(\rho(y,z)u,v)x+\mu(u,\rho(y,z)v)x;\label{eq:mp3}\\
~\rho(x,y)[u,v,w]_{\g_2}&=&[\rho(x,y)u,v,w]_{\g_2}+[u,\rho(x,y)v,w]_{\g_2}+[u,v,\rho(x,y)w]_{\g_2};\label{eq:mp4}\\
~[u,v,\rho(x,y)w]_{\g_2}&=&\rho(\mu(u,v)x,y)w-\rho (\mu(u,w)y,x)v+\rho (\mu(v,w)y,x)u;\label{eq:mp5}\\
~[\rho(x,y)u,v,w]_{\g_2}&=&\rho(x,y)[u,v,w]_{\g_2}+\rho(\mu(v,w)x,y)u+\rho(x,\rho(v,w)y)u,\label{eq:mp6}
\end{eqnarray}
 then we call $(\g_1,\g_2;\rho,\mu)$ a {\bf matched pair of $3$-Lie algebras}.
\end{defi}

\begin{defi}{\rm (\cite{Jiefeng})}
A {\bf generalized representation}  of a $3$-Lie algebra $(\g,\pi=[\cdot,\cdot,\cdot]_{\g})$ on a vector space $V$ consists of linear
maps: $\rho:\wedge^{2}\g\rightarrow \gl(V)$ and $\nu:\g\rightarrow \Hom(\wedge^{2}V,V)$, such that
\begin{eqnarray}
~\label{generalized-representation-1}
[\hat{\pi}+\hat{\rho}+\hat{\nu},\hat{\pi}+\hat{\rho}+\hat{\nu}]_{\Li}=0.
\end{eqnarray}
\emptycomment{
where $\bar{\rho}:\wedge^{3}(\g\oplus V)\rightarrow \g\oplus V$ is defined by
\begin{eqnarray*}
~\label{generalized-representation-2}
~~~\bar{\rho}(x+u,u+v,z+w)=\rho(x,y)(w)+\rho(y,z)(u)+\rho(z,x)(v),\quad \forall x,y,z\in \g,u,v,w\in V,
\end{eqnarray*}
and $\bar{\nu}:\wedge^{3}(\g\oplus V)\rightarrow \g\oplus V$ is induced by $\nu$ via
\begin{eqnarray*}
~\label{generalized-representation-3}
~~~\bar{\nu}(x+u,u+v,z+w)=\nu(x)(v,w)+\nu(y)(w,u)+\nu(z)(u,v),\quad \forall x,y,z\in \g,u,v,w\in V.
\end{eqnarray*}
}
\end{defi}

We will refer to a generalized representation by a triple $(V;\rho,\nu).$

\emptycomment{
\begin{pro}{\rm (\cite{Jiefeng})}
Linear maps $\rho:\wedge^2\g\rightarrow\End(V)$ and $\nu:\g\rightarrow\Hom(\wedge^2 V,V)$ give rise to a generalized representation of a $3$-Lie algebra
$\g$ on a vector space $V$ if and only if for all $x_{i}\in \g, v_i\in V,$ the following equalities hold:
\begin{eqnarray}
\label{eq:r1}\rho(x_1,x_2)\rho(x_3,x_4)&=&\rho([x_1,x_2,x_3],x_4)+\rho(x_3,[x_1,x_2,x_4])+\rho(x_3,x_4)\rho(x_1,x_2),\\
\label{eq:r2}\rho(x_1,[x_2,x_3,x_4])&=&\rho(x_3,x_4)\rho(x_1,x_2)-\rho(x_2,x_4)\rho(x_1,x_3)+\rho(x_2,x_3)\rho(x_1,x_4),\\
\nonumber\rho(x_1,x_2)(\nu(x_3)(v_1,v_2))&=&\nu([x_1,x_2,x_3])(v_1,v_2)+\nu(x_3)(\rho(x_1,x_2)(v_1),v_2)\\
\label{eq:r3}&&+\nu(x_3)(v_1,\rho(x_1,x_2)(v_2)),\\
\nonumber\nu(x_1)(v_1,\rho(x_2,x_3)(v_2))&=&\nu(x_3)(v_2,\rho(x_2,x_1)(v_1))+\nu(x_2)(\rho(x_3,x_1)(v_1),v_2)\\
\label{eq:r4}&&+\rho(x_2,x_3)(\nu(x_1)(v_1,v_2)),\\
\label{eq:r6}\nu(x_1)(v_1,\nu(x_2)(v_2,v_3))&=&\nu(x_2)(\nu(x_1)(v_1,v_2),v_3)+\nu(x_2)(v_2,\nu(x_1)(v_1,v_3)),\\
\label{eq:r7}\nu(x_1)(\nu(x_2)(v_1,v_2),v_3)&=&\nu(x_2)(\nu(x_1)(v_1,v_2),v_3).
\end{eqnarray}
\end{pro}
}

\begin{pro}\label{Generalized matched pair}
Let $(\g_1,[\cdot,\cdot,\cdot]_{\g_1})$ and $(\g_2,[\cdot,\cdot,\cdot]_{\g_2})$ be two $3$-Lie algebras. Suppose that there are linear maps $\rho:\wedge^2\g_1\rightarrow\gl(\g_2), \nu:\g_1\rightarrow \Hom(\wedge^2\g_2,\g_2)$ and $\varrho:\wedge^2\g_2\rightarrow\gl(\g_1), \tau:\g_2\rightarrow\Hom(\wedge^2\g_1,\g_1)$ satisfying the following conditions:
\begin{itemize}
\item[{\rm (a)}]  $(\g_2;\rho,\nu)$ is a generalized representation of $\g_1;$
 
\item[{\rm (b)}]  $(\g_1;\varrho,\tau)$ is a generalized representation of $\g_2;$ 

\item[{\rm (c)}]  For all $x_i\in \g_1$ and $a_i\in \g_2, 1\leq i \leq 4,$~$\rho$, $\nu$, $\varrho$ and $\tau$ satisfy the following conditions:
 {\footnotesize
\begin{eqnarray}
\label{eq:mp-1}~\nu([x_1,x_2,x_3]_{\g_1})(a_1\wedge a_2)&=&\rho(x_1,x_2)(\nu(x_3)(a_1\wedge a_2))+\rho(x_3,x_1)(\nu(x_2)(a_1\wedge a_2))\\
\nonumber&&+\rho(x_2,x_3)(\nu(x_1)(a_1\wedge a_2));\\
\label{eq:mp-2}~\varrho(a_1,a_2)[x_1,x_2,x_3]_{\g_1}&=&[\varrho(a_1,a_2)x_1,x_2,x_3]_{\g_1}+[x_1,\varrho(a_1,a_2)x_2,x_3]_{\g_1}+[x_1,x_2,\varrho(a_1,a_2)x_3]_{\g_1}\\
\nonumber&&~+\tau(\nu(x_1)(a_1\wedge a_2))(x_2\wedge x_3)+\tau(\nu(x_2)(a_1\wedge a_2))(x_3\wedge x_1)\\
\nonumber&&+\tau(\nu(x_3)(a_1\wedge a_2))(x_1\wedge x_2);\\
\label{eq:mp-3}~[x_1,x_2,\varrho(a_1,a_2)x_3]_{\g_1}&=&\varrho(\rho(x_1,x_2)a_1,a_2)x_3-\varrho (\rho(x_1,x_3)a_2,a_1)x_2+\varrho (\rho(x_2,x_3)a_2,a_1)x_1\\
\nonumber&&-\tau(\nu(x_3)(a_1\wedge a_2))(x_1\wedge x_2);\\
\label{eq:mp-4}~[\varrho(a_1,a_2)x_1,x_2,x_3]_{\g_1}&=&\varrho(a_1,a_2)[x_1,x_2,x_3]_{\g_1}+\varrho(\rho(x_2,x_3)a_1,a_2)x_1+\varrho(a_1,\rho(x_2,x_3)a_2)x_1\\
\nonumber&&-\tau(\nu(x_1)(a_1\wedge a_2))(x_2\wedge x_3);\\
\label{eq:mp-5}~[x_1,x_2,\tau(a_1)(x_3\wedge x_4)]_{\g_1}&=&[x_3,x_4,\tau(a_1)(x_1\wedge x_2)]_{\g_1}-\tau(\rho(x_3,x_4)a_1)(x_1\wedge x_2)\\
\nonumber&&+\tau(a_1)([x_1,x_2,x_3]_{\g_1}\wedge x_4)+\tau(a_1)(x_3\wedge [x_1,x_2,x_4]_{\g_1})\\
\nonumber&&+\tau(\rho(x_1,x_2)a_1)(x_3\wedge x_4);\\
\label{eq:mp-6}~\tau(a_1)(x_1\wedge[x_2,x_3,x_4]_{\g_1})&=&[\tau(a_1)(x_1\wedge x_2),x_3,x_4]_{\g_1}+[\tau(a_1)(x_1\wedge x_3),x_4,x_2]_{\g_1}\\
\nonumber&&+[\tau(a_1)(x_1\wedge x_4),x_2,x_3]_{\g_1}+\tau(\rho(x_1,x_2)a_1)(x_3\wedge x_4)\\ \nonumber&&+\tau(\rho(x_1,x_3)a_1)(x_4\wedge x_2)+\tau(\rho(x_1,x_4)a_1)(x_2\wedge x_3);\\
\label{eq:mp-7}~\rho(\tau(a_2)(x_1\wedge x_2),x_3)a_1&=&\rho(\tau(a_2)(x_3\wedge x_2),x_1)a_1+\rho(\tau(a_2)(x_1\wedge x_3),x_2)a_1;\\
\label{eq:mp-8}~\rho(x_1,\tau(a_1)(x_2\wedge x_3))(a_2)&=&\rho(x_1,\tau(a_2)(x_1\wedge x_3))(a_1);\\
\label{eq:mp-9}~\tau([a_1,a_2,a_3]_{\g_2})(x_1\wedge x_2)&=&\varrho(a_1,a_2)(\tau(a_3)(x_1\wedge x_2))+\varrho(a_3,a_1)(\tau(a_2)(x_1\wedge x_2))\\
\nonumber&&+\varrho(a_2,a_3)(\tau(a_1)(x_1\wedge x_2));\\
\label{eq:mp-10}~\rho(x_1,x_2)[a_1,a_2,a_3]_{\g_2}&=&[\rho(x_1,x_2)a_1,a_2,a_3]_{\g_2}+[a_1,\rho(x_1,x_2)a_2,w]_{\g_2}\\
\nonumber&&+[a_1,a_2,\rho(x_1,x_2)a_3]_{\g_2}+\nu(\tau(a_1)(x_1\wedge x_2))(a_2\wedge a_3)\\
\nonumber&&+\nu(\tau(a_2)(x_1\wedge x_2))(a_3\wedge a_1)+\nu(\tau(a_3)(x_1\wedge x_2))(a_1\wedge a_2);\\
\label{eq:mp-11}~[a_1,a_2,\rho(x_1,x_2)a_3]_{\g_2}&=&\rho(\varrho(a_1,a_2)x_1,x_2)a_3-\rho (\varrho(a_1,a_3)x_2,x_1)a_2+\rho (\varrho(a_2,a_3)x_2,x_1)a_1\\
\nonumber&&-\nu(\tau(a_3)(x_1\wedge x_2))(a_1\wedge a_2);\\
\label{eq:mp-12}~[\rho(x_1,x_2)a_1,a_2,a_3]_{\g_2}&=&\rho(x_1,x_2)[a_1,a_2,a_3]_{\g_2}+\rho(\varrho(a_2,a_3)x_1,x_2)a_1+\rho(x_1,\varrho(a_2,a_3)x_2)a_1\\
\nonumber&&-\nu(\tau(a_1)(x_1\wedge x_2))(a_2\wedge a_3);\\
\label{eq:mp-13}~[a_1,a_2,\nu(x_1)(a_3\wedge a_4)]_{\g_2}&=&\nu(x_1)([a_1,a_2,a_3]_{\g_2}\wedge a_4)+\nu(x_1)(a_3\wedge [a_1,a_2,a_4]_{\g_2})\\
\nonumber&&+[a_3,a_4,\nu(x_1)(a_1\wedge a_2)]_{\g_2}+\nu(\varrho(a_1,a_2)x_1)(a_3\wedge a_4)\\
\nonumber&&-\nu(\varrho(a_3,a_4)x_1)(a_1\wedge a_2);\\
\label{eq:mp-14}~\nu(x_1)(a_1\wedge [a_2,a_3,a_4]_{\g_2})&=&\nu(\varrho(a_1,a_2)x_1)(a_3\wedge a_4)+\nu(\varrho(a_1,a_3)x_1)(a_1\wedge a_2)\\
\nonumber&&+\nu(\varrho(a_1,a_4)x_1)(a_2\wedge a_3)+[\nu(x_1)(a_1\wedge a_2),a_3,a_4]_{\g_2}\\
\nonumber&&+[\nu(a_1)(a_1\wedge a_3),a_4,a_2]_{\g_2}+[a_2,a_3,\nu(x_1)(a_1\wedge a_4)]_{\g_2};\\
\label{eq:mp-15}~\varrho(\nu(x_2)(a_1\wedge a_2),a_3)x_1&=&\varrho(\nu(x_2)(a_3\wedge a_2),a_1)x_1+\varrho(\nu(x_2)(a_1\wedge a_3),a_2)x_1;\label{eq:mp14}\\
\label{eq:mp-16}~\varrho(a_1,\nu(x_1)(a_2\wedge a_3))(x_2)&=&\varrho(a_1,\nu(x_2)(a_1\wedge a_3))(x_1).
\end{eqnarray}}
\end{itemize}
Then there is a $3$-Lie algebra structure on $\g_1\oplus\g_2$ (as the direct sum of vector spaces) defined by
\begin{eqnarray}\label{eq:formula}
~&&[x_1+a_1,x_2+a_2,x_2+a_3]_{\g_1\oplus\g_2}\\
\nonumber&=&[x_1,x_2,x_3]_{\g_1}+\rho(x_1,x_2)(a_3)+\rho(x_2,x_3)(a_1)+\rho(x_3,x_1)(a_2)\\
\nonumber&&+\nu(x_1)(a_2\wedge a_3)+\nu(x_2)(a_3\wedge a_1)+\nu(x_3)(a_1\wedge a_2)\\
\nonumber&&+[a_1,a_2,a_3]_{\g_2}+\varrho(a_1,a_2)(x_3)+\varrho(a_2,a_3)(x_1)+\varrho(a_3,a_1)(x_2)\\
\nonumber&&+\tau(a_1)(x_2\wedge x_3)+\tau(a_2)(x_3\wedge x_1)+\tau(a_3)(x_1\wedge x_2),
\end{eqnarray}
for all $x_i\in \g_1,a_i\in \g_2, 1\leq i\leq 3.$
\end{pro}
\begin{proof}
It follows from straightforward applications of the Fundamental Identity for the bracket operation $[\cdot,\cdot,\cdot]_{\g_1\oplus\g_2}.$
\end{proof}

\begin{defi}
 Let $(\g_1,[\cdot,\cdot,\cdot]_{\g_1})$ and $(\g_2,[\cdot,\cdot,\cdot]_{\g_2})$ be $3$-Lie algebras. Suppose that there are linear maps $\rho:\wedge^2\g_1\rightarrow\gl(\g_2), \nu:\g_1\rightarrow \Hom(\wedge^2\g_2,\g_2)$ and $\varrho:\wedge^2\g_2\rightarrow\gl(\g_1), \tau:\g_2\rightarrow\Hom(\wedge^2\g_1,\g_1)$ such that $(\g_2;\rho,\nu)$ is a generalized representation of $\g_1$, $(\g_1;\varrho,\tau)$ is a generalized representation of $\g_2$
 and $\rho, \nu, \varrho, \tau$ satisfy Eqs. \eqref{eq:mp-1}-\eqref{eq:mp-16}. Then we call $(\g_1,\g_2;\rho,\nu,\varrho,\tau)$ a {\bf generalized matched pair of $3$-Lie algebras}.
\end{defi}

\begin{thm}
  Let  $(\g_1,\g_2;\rho,\nu,\varrho,\tau)$ be a generalized matched pair of $3$-Lie algebras. Then $((\g_1\oplus \g_2, [\cdot,\cdot,\cdot]_{\g_1\oplus\g_2}),\g_1,\g_2)$ is a twilled $3$-Lie algebra.

\end{thm}

\begin{proof}
Let $(\g_1,\g_2;\rho,\nu,\varrho,\tau)$ be a generalized matched pair of $3$-Lie algebras. We denote the $3$-Lie bracket $[\cdot,\cdot,\cdot]_{\g_{1}}$ and $[\cdot,\cdot,\cdot]_{\g_{2}}$ by $\pi_{1}$ and $\pi_{2}$ respectively. By Proposition \ref{Generalized matched pair},   $(\g_1\oplus\g_2,[\cdot,\cdot,\cdot]_{\g_1\oplus\g_2})$ is a  $3$-Lie algebra. Moreover for all $x_i\in \g_1,a_i\in \g_2, 1\leq i\leq 3,$ $$[x_1+a_1,x_2+a_2,x_2+a_3]_{\g_1\oplus\g_2}=(\hat{\pi}_{1}+\hat{\rho}+\hat{\nu}+\hat{\pi}_{2}+\hat{\varrho}+\hat{\tau})(x_1+a_1,x_2+a_2,x_2+a_3).$$
Therefore, by Proposition \ref{pro:3LieMC}, we have
$$[\hat{\pi}_{1}+\hat{\rho}+\hat{\nu}+\hat{\pi}_{2}+\hat{\varrho}+\hat{\tau},\hat{\pi}_{1}+\hat{\rho}+\hat{\nu}+\hat{\pi}_{2}+\hat{\varrho}+\hat{\tau}]_{\Li}=0,$$
which implies that $((\g_1\oplus \g_2, [\cdot,\cdot,\cdot]_{\g_1\oplus\g_2}),\g_1,\g_2)$ is a twilled $3$-Lie algebra.
\end{proof}

The converse of the above theorem is not true in general. At the end of this section, we introduce the notion of a strict twilled 3-Lie algebra and show that it is equivalent to a (usual) matched pair of $3$-Lie algebras.

\begin{defi}
A twilled $3$-Lie algebra $(\huaG,\g_1,\g_2)$ is called   a {\bf strict twilled $3$-Lie algebra} if $\psi=0.$
 \end{defi}

 By Lemma \ref{proto-twilled}, we have the following corollary.

\begin{cor}\label{lem:twilled-strict-t}
The triple $(\huaG,\g_1,\g_2)$ is a strict twilled $3$-Lie algebra if and only if the following conditions hold:
\begin{eqnarray}
\label{strict-twilled-1}\frac{1}{2}[\hat{\mu}_1,\hat{\mu}_1]_{\Li}&=&0,\\
\label{strict-twilled-2}[\hat{\mu}_1,\hat{\mu}_2]_{\Li}&=&0,\\
\label{strict-twilled-3}\frac{1}{2}[\hat{\mu}_2,\hat{\mu}_2]_{\Li}&=&0.
\end{eqnarray}
\end{cor}

\begin{pro}
  There is a one-to-one correspondence between matched pairs of $3$-Lie algebras and strict twilled $3$-Lie algebras.
\end{pro}

\begin{proof}
Let $(\g_1,\g_2;\rho,\mu)$ be a matched pair of $3$-Lie algebras. It is straightforward to deduce that $(\g_1\oplus\g_2,[\cdot,\cdot,\cdot]_{\g_1\oplus\g_2})$ is a $3$-Lie algebra, where the bracket operation  $[\cdot,\cdot,\cdot]_{\g_1\oplus\g_2}$ on the direct sum vector space $\g_1\oplus\g_2$ is given by
\begin{eqnarray}\label{eq:dm}
~[x+u,y+v,z+w]_{\g_1\oplus\g_2}&=&[x,y,z]_{\g_1}+\rho(x,y)w+\rho(y,z)u+\rho(z,x)v\\
\nonumber&&+[u,v,w]_{\g_2}+\mu(u,v)z+\mu(v,w)x+\mu(w,u)y.
\end{eqnarray}
We denote this $3$-Lie algebra simply by $\g_1\bowtie\g_2$. Then $(\g_1\bowtie\g_2,\g_1,\g_2)$ is a strict twilled $3$-Lie algebra.

Conversely, if $(\huaG,\g_1,\g_2)$ is a strict twilled $3$-Lie algebra, by \eqref{strict-twilled-1}, we deduce that $\hat{\mu}_1$ is a $3$-Lie algebra structure on $\huaG=\g_1\oplus\g_2$. Moreover, by \eqref{bracket-2}, we obtain that $\hat{\mu}_1:\wedge^3\g_1\rightarrow\g_1$ is a $3$-Lie bracket, which we denote by $(\g_1,[\cdot,\cdot,\cdot]_{\g_1})$ and $\rho(x,y)w:=\hat{\mu}_1(x,y,w)$ is a representation of $(\g_1,[\cdot,\cdot,\cdot]_{\g_1})$ on the vector space $\g_2$.

Similarly, we have $\hat{\mu}_2:\wedge^3\g_2\rightarrow\g_2$ is a $3$-Lie bracket, which we denote by $(\g_2,[\cdot,\cdot,\cdot]_{\g_2})$ and $\mu(v,w)x:=\hat{\mu}_2(v,w,x)$ is a representation of $(\g_2,[\cdot,\cdot,\cdot]_{\g_2})$ on the vector space $\g_1$.

By $[\hat{\mu}_1,\hat{\mu}_2]_{\Li}=0,$ for all $x,y,z\in\g_1,u,v\in \g_2,$ we have
\begin{eqnarray*}
[\hat{\mu}_1,\hat{\mu}_2]_{\Li}(u,v,x,y,z)&=&(\hat{\mu}_1{\circ}\hat{\mu}_2+\hat{\mu}_2{\circ}\hat{\mu}_1)(u,v,x,y,z)\\
                                  &=&\hat{\mu}_1(\hat{\mu}_2(u,v,x),y,z)+\hat{\mu}_1(x,\hat{\mu}_2(u,v,y),z)-\hat{\mu}_1(u,v,\hat{\mu}_2(x,y,z))\\
                                  &&+\hat{\mu}_1(x,y,\hat{\mu}_2(u,v,z))+\hat{\mu}_2(\hat{\mu}_1(u,v,x),y,z)+\hat{\mu}_2(x,\hat{\mu}_1(u,v,y),z)\\
                                  &&-\hat{\mu}_2(u,v,\hat{\mu}_1(x,y,z))+\hat{\mu}_2(x,y,\hat{\mu}_1(u,v,z))\\
                                  &=&[\mu(u,v)x,y,z]_{\g_1}+[x,\mu(u,v)y,z]_{\g_1}+[x,y,\mu(u,v)z]_{\g_1}-\mu(u,v)[x,y,z]_{\g_1}\\
                                  &=&0.
\end{eqnarray*}
Thus, we deduce that \eqref{eq:mp1} holds. Moreover, we have
\begin{eqnarray*}
&&[\hat{\mu}_1,\hat{\mu}_2]_{\Li}(z,v,u,y,x)\\
&&=-(-[x,y,\mu(u,v)z]_{\g_1}+\mu(\rho(x,y)u,v)z-\mu (\rho(x,z)v,u)y+\mu (\rho(y,z)v,u)x)=0,\\
&&[\hat{\mu}_1,\hat{\mu}_2]_{\Li}(y,z,u,v,x)\\
&&=-[\mu(u,v)x,y,z]_{\g_1}+\mu(\rho(y,z)u,v)x+\mu(u,\rho(y,z)v)x+\mu(u,v)[x,y,z]_{\g_1}=0.
\end{eqnarray*}
Thus, we deduce that \eqref{eq:mp2} and \eqref{eq:mp3} hold. Similarly, we can deduce that \eqref{eq:mp4}-\eqref{eq:mp6} hold.
Thus, $(\g_1,\g_2;\rho,\mu)$ is a matched pair of $3$-Lie algebras.
The proof is finished.
\end{proof}

\section{Twisting of twilled 3-Lie algebras by Maurer-Cartan elements}\label{sec:T}

In this section, we introduce the notion of a twisting of a 3-Lie algebra. In particular, we show that the twisting of a twilled 3-Lie algebra by a Maurer-Cartan element of the associated $L_\infty$-algebra given in Theorem \ref{quasi-as-shLie} is still a twilled 3-Lie algebra.

Let $(\huaG,[\cdot,\cdot,\cdot]_\huaG)$ be a $3$-Lie algebra with a decomposition into two subspaces $\huaG=\g_1\oplus\g_2$ equipped with the structure $\Omega=\hat{\phi}_1+\hat{\mu}_1+\hat{\psi}+\hat{\mu}_2+\hat{\phi}_2$.
Let $\hat{H}$ be the lift of a linear map $H:\g_2\lon\g_1$. Since $[\cdot,\hat{H}]_{\Li}$ is an inner derivation of the graded Lie algebra $(C^*_{\Li}(\huaG,\huaG),[\cdot,\cdot]_{\Li})$, $e^{[\cdot,\hat{H}]_{\Li}}$ is an automorphism of the graded Lie algebra $(C^*_{\Li}(\huaG,\huaG),[\cdot,\cdot]_{\Li})$.

\begin{defi}
The transformation $\Omega^{H}:=e^{[\cdot,\hat{H}]_{\Li}}\Omega$ is called a {\bf twisting} of $\Omega$ by $H$.
\end{defi}

The precise relation between $\Omega^{H}$ and $\Omega$ is given as follows.

\begin{lem}
$\Omega^{H}=e^{-\hat{H}}\circ \Omega\circ (e^{\hat{H}}\otimes e^{\hat{H}}\otimes e^{\hat{H}})$.
\end{lem}
\begin{proof}
For all $(x_1,v_1),~(x_2,v_2),~(x_3,v_3) \in\huaG$, we have
\begin{eqnarray*}
&&[\Omega,\hat{H}]_{\Li}\big((x_1,v_1),(x_2,v_2),(x_3,v_3)\big)\\
&=&(\Omega{\circ}\hat{H}-\hat{H}{\circ}\Omega)\big((x_1,v_1),(x_2,v_2),(x_3,v_3)\big)\\
&=&\Omega((H(v_1),0),(x_2,v_2),(x_3,v_3))+\Omega((x_1,v_1),(H(v_2),0),,(x_3,v_3))\\
&&\quad+\Omega((x_1,v_1),(x_2,v_2),(H(v_3),0))-\hat{H}(\Omega((x_1,v_1),(x_2,v_2),(x_3,v_3))).
\end{eqnarray*}
By $\hat{H}\circ\hat{H}=0$, we have
\begin{eqnarray*}
&&[[\Omega,\hat{H}]_{\Li},\hat{H}]_{\Li}\big((x_1,v_1),(x_2,v_2),(x_3,v_3)\big)\\
&=&[\Omega,\hat{H}]_{\Li}((H(v_1),0),(x_2,v_2),(x_3,v_3))+[\Omega,\hat{H}]_{\Li}((x_1,v_1),(H(v_2),0),(x_3,v_3))\\
&&+[\Omega,\hat{H}]_{\Li}((x_1,v_1),(x_2,v_2),(H(v_3),0))-\hat{H}\big([\Omega,\hat{H}]_{\Li}((x_1,v_1),(x_2,v_2),(x_3,v_3))\big)\\
&=&2\Omega((H(v_1),0),(H(v_2),0),(x_3,v_3))+2\Omega((H(v_1),0),(x_2,v_2),(H(v_3),0))\\
&&+2\Omega((x_1,v_1),(H(v_2),0),(H(v_3),0))-2\hat{H}\Omega((H(v_1),0),(x_2,v_2),(x_3,v_3))\\
&&-2\hat{H}\Omega((x_1,v_1),(H(v_2),0),(x_3,v_3))-2\hat{H}\Omega((x_1,v_1),(x_2,v_2),(H(v_3),0)).
\end{eqnarray*}
Similarly, we have
\begin{eqnarray*}
&&[[[\Omega,\hat{H}]_{\Li},\hat{H}]_{\Li},\hat{H}]_{\Li}\big((x_1,v_1),(x_2,v_2),(x_3,v_3)\big)\\
&=&6\Omega((H(v_1),0),(H(v_2),0),(H(v_3),0))-6\hat{H}(\Omega((H(v_1),0),(H(v_2),0),(x_3,v_3)))\\
&&\quad-6\hat{H}(\Omega((H(v_1),0),(x_2,v_2),(H(v_3),0)))-6\hat{H}(\Omega((x_1,v_1),(H(v_2),0),(H(v_3),0)))
\end{eqnarray*}
and
\begin{eqnarray*}
&&[[[[\Omega,\hat{H}]_{\Li},\hat{H}]_{\Li},\hat{H}]_{\Li},\hat{H}]_{\Li}\big((x_1,v_1),(x_2,v_2),(x_3,v_3)\big)\\
&=&-24\hat{H}(\Omega ((H(v_1),0),(H(v_2),0),(H(v_3),0))).
\end{eqnarray*}
Moreover, for all $ i\ge5$, $$(([\cdot,\hat{H}]_{\Li})^i\Omega)\big((x_1,v_1),(x_2,v_2),(x_3,v_3)\big)=0.$$
Therefore,  we have
\begin{equation}
\begin{aligned}\label{twisting-operator}
\Omega^{H}=&e^{[\cdot,\hat{H}]_{\Li}}\Omega\\
=&\Omega+[\Omega,\hat{H}]_{\Li}+\half[[\Omega,\hat{H}]_{\Li},\hat{H}]_{\Li}+\frac{1}{6}[[[\Omega,\hat{H}]_{\Li},\hat{H}]_{\Li},\hat{H}]_{\Li}\\
&+\frac{1}{24}[[[[\Omega,\hat{H}]_{\Li},\hat{H}]_{\Li},\hat{H}]_{\Li},\hat{H}]_{\Li}\\
=&\Omega+\Omega\circ(\hat{H}\otimes{\Id}\otimes{\Id})+\Omega\circ({\Id}\otimes\hat{H}\otimes{\Id})+\Omega\circ({\Id}\otimes{\Id}\otimes\hat{H})
-\hat{H}\circ\Omega\\
&+\Omega\circ(\hat{H}\otimes\hat{H}\otimes{\Id})+\Omega\circ(\hat{H}\otimes{\Id}\otimes\hat{H})+\Omega\circ({\Id}\otimes\hat{H}\otimes\hat{H})\\
&-\hat{H}\circ\Omega\circ(\hat{H}\otimes{\Id}\otimes{\Id})-\hat{H}\circ\Omega\circ({\Id}\otimes\hat{H}\otimes{\Id})-\hat{H}\circ\Omega\circ({\Id}\otimes{\Id}\otimes\hat{H})\\
&+\Omega\circ(\hat{H}\otimes\hat{H}\otimes\hat{H})-\hat{H}\circ\Omega\circ(\hat{H}\otimes\hat{H}\otimes\hat{H})\\
&-\hat{H}\circ\Omega\circ(\hat{H}\otimes\hat{H}\otimes{\Id})-\hat{H}\circ\Omega\circ(\hat{H}\otimes{\Id}\otimes\hat{H})-\hat{H}\circ\Omega\circ({\Id}\otimes\hat{H}\otimes\hat{H})\\
=&(\Id-\hat{H})\circ\Omega\circ((\Id+\hat{H})\otimes(\Id+\hat{H})\otimes(\Id+\hat{H}))\\
=&e^{-\hat{H}}\circ \Omega\circ (e^{\hat{H}}\otimes e^{\hat{H}}\otimes e^{\hat{H}}).
\end{aligned}
\end{equation}
 The proof is finished.
\end{proof}

\begin{pro}
The twisting $\Omega^{H}$ is a $3$-Lie algebra structure on $\huaG$, i.e. one has $$[\Omega^{H},\Omega^{H}]_{\Li}=0.$$
\end{pro}
\begin{proof}
By $\Omega^{H}=e^{-\hat{H}}\circ \Omega\circ (e^{\hat{H}}\otimes e^{\hat{H}}\otimes e^{\hat{H}})$, we have
\begin{eqnarray*}
[\Omega^{H},\Omega^{H}]_{\Li}=2\Omega^{H}{\circ}\Omega^{H}&=&2e^{-\hat{H}}\circ(\Omega{\circ}\Omega)\circ(e^{\hat{H}}\otimes e^{\hat{H}}\otimes e^{\hat{H}}\otimes e^{\hat{H}}\otimes e^{\hat{H}})\\
&=&e^{-\hat{H}}\circ[\Omega,\Omega]_{\Li}\circ(e^{\hat{H}}\otimes e^{\hat{H}}\otimes e^{\hat{H}}\otimes e^{\hat{H}}\otimes e^{\hat{H}})=0.
\end{eqnarray*}
The proof is finished.
\end{proof}

\begin{cor}\label{twisting-isomorphism}
$
e^{\hat{H}}:(\huaG,\Omega^{H})\lon(\huaG,\Omega)
$
is an isomorphism between $3$-Lie algebras.
\end{cor}

The twisting operations are completely determined by the following result.

\begin{pro}\label{thm:twist}
Write $\Omega:=\hat{\phi}_1+\hat{\mu}_1+\hat{\psi}+\hat{\mu}_2+\hat{\phi}_2$ and $\Omega^{H}:=\hat{\phi}_1^{H}+\hat{\mu}_1^{H}+\hat{\psi}^{H}+\hat{\mu}_2^{H}+\hat{\phi}_2^{H}$. Then we have:
\begin{eqnarray}
\label{twisting-1}\hat{\phi}_1^{H}&=&\hat{\phi}_1,\\
\label{twisting-2}\hat{\mu}_1^{H}&=&\hat{\mu}_1+[\hat{\phi}_1,\hat{H}]_{\Li},\\
\label{twisting-3}\hat{\psi}^{H}&=&\hat{\psi}+[\hat{\mu}_1,\hat{H}]_{\Li}+\half[[\hat{\phi}_1,\hat{H}]_{\Li},\hat{H}]_{\Li},\\
\label{twisting-4}\hat{\mu}_2^{H}&=&\hat{\mu}_2+[\hat{\psi},\hat{H}]_{\Li}+\half[[\hat{\mu}_1,\hat{H}]_{\Li},\hat{H}]_{\Li}+\frac{1}{6}[[[\hat{\phi}_1,\hat{H}]_{\Li},\hat{H}]_{\Li},\hat{H}]_{\Li},\\
\label{twisting-5}\hat{\phi}_2^{H}&=&\hat{\phi}_2+[\hat{\mu}_2,\hat{H}]_{\Li}+\half[[\hat{\psi},\hat{H}]_{\Li},\hat{H}]_{\Li}+\frac{1}{6}[[[\hat{\mu}_1,\hat{H}]_{\Li},\hat{H}]_{\Li},\hat{H}]_{\Li}\\
\nonumber&&+\frac{1}{24}[[[[\hat{\phi}_1,\hat{H}]_{\Li},\hat{H}]_{\Li},\hat{H}]_{\Li},\hat{H}]_{\Li}.
\end{eqnarray}
\end{pro}

\begin{proof}
By \eqref{twisting-operator}, the first term of $\Omega^{H}$ is $\Omega$. By Lemma \ref{Zero-condition-2} and $||\hat{\phi}_2||=-1|3,~||\hat{H}||=-1|1$, we have
$$
[\Omega,\hat{H}]_{\Li}=[\hat{\phi}_1,\hat{H}]_{\Li}+[\hat{\mu}_1,\hat{H}]_{\Li}+[\hat{\psi},\hat{H}]_{\Li}+[\hat{\mu}_2,\hat{H}]_{\Li}.
$$
By Lemma \ref{important-lemma-2} and $||\hat{\phi}_1||=3|-1,~||\hat{\mu}_1||=2|0,~||\hat{\psi}||=1|1,~||\hat{\mu}_2||=0|2$, we have $$||[\hat{\phi}_1,\hat{H}]_{\Li}||=2|0,\quad ||[\hat{\mu}_1,\hat{H}]_{\Li}||=1|1,\quad||[\hat{\psi},\hat{H}]_{\Li}||=0|2,\quad||[\hat{\mu}_2,\hat{H}]_{\Li}||=-1|3.$$
Therefore, $[[\hat{\mu}_2,\hat{H}]_{\Li},\hat{H}]_{\Li}=0$ and
$$
\half[[\Omega,\hat{H}]_{\Li},\hat{H}]_{\Li}=\half([[\hat{\phi}_1,\hat{H}]_{\Li},\hat{H}]_{\Li}+[[\hat{\mu}_1,\hat{H}]_{\Li},\hat{H}]_{\Li}+[[\hat{\psi},\hat{H}]_{\Li},\hat{H}]_{\Li}).
$$
Moreover, we have
$$||[[\hat{\phi}_1,\hat{H}]_{\Li},\hat{H}]_{\Li}||=1|1,\quad||[[\hat{\mu}_1,\hat{H}]_{\Li},\hat{H}]_{\Li}||=0|2,\quad||[[\hat{\psi},\hat{H}]_{\Li},\hat{H}]_{\Li}||=-1|3.$$
Thus, $[[[\hat{\psi},\hat{H}]_{\Li},\hat{H}]_{\Li},\hat{H}]_{\Li}=0$ and
$$
\frac{1}{6}[[[\Omega,\hat{H}]_{\Li},\hat{H}]_{\Li},\hat{H}]_{\Li}=\frac{1}{6}([[[\hat{\phi}_1,\hat{H}]_{\Li},\hat{H}]_{\Li},\hat{H}]_{\Li}+[[[\hat{\mu}_1,\hat{H}]_{\Li},\hat{H}]_{\Li},\hat{H}]_{\Li}).
$$
By
$
||[[[\hat{\phi}_1,\hat{H}]_{\Li},\hat{H}]_{\Li},\hat{H}]_{\Li}||=0|2$ and $||[[[\hat{\mu}_1,\hat{H}]_{\Li},\hat{H}]_{\Li},\hat{H}]_{\Li}||=-1|3,
$
we have $$[[[[\hat{\mu}_1,\hat{H}]_{\Li},\hat{H}]_{\Li},\hat{H}]_{\Li},\hat{H}]_{\Li}=0.$$
Therefore,
$$\frac{1}{24}[[[[\Omega,\hat{H}]_{\Li},\hat{H}]_{\Li},\hat{H}]_{\Li},\hat{H}]_{\Li}=\frac{1}{24}[[[[\hat{\phi}_1,\hat{H}]_{\Li},\hat{H}]_{\Li},\hat{H}]_{\Li},\hat{H}]_{\Li}.$$

It is straightforward to see that the terms of the bidegree $-1|3$ are given by
\begin{eqnarray*}
&&\hat{\phi}_2+[\hat{\mu}_2,\hat{H}]_{\Li}+\half[[\hat{\psi},\hat{H}]_{\Li},\hat{H}]_{\Li}+\frac{1}{6}[[[\hat{\mu}_1,\hat{H}]_{\Li},\hat{H}]_{\Li},\hat{H}]_{\Li}\\
 &&+\frac{1}{24}[[[[\hat{\phi}_1,\hat{H}]_{\Li},\hat{H}]_{\Li},\hat{H}]_{\Li},\hat{H}]_{\Li}.
\end{eqnarray*}
Thus, we deduce that \eqref{twisting-5} holds. Similarly, by considering the terms of the bidegrees $3|-1$, $2|0$, $1|1$, $0|2$,
  we deduce  that \eqref{twisting-1}-\eqref{twisting-4} hold respectively.
\end{proof}

Now we consider the twisting of twilled 3-Lie algebras.

\begin{thm}\label{Twilled-3-Lie}
Let $(\huaG,\g_1,\g_2)$ be a twilled $3$-Lie algebra and $H:\g_2\longrightarrow\g_1$ a linear map. Then   the twisting of $((\huaG,\Omega^{H}),\g_1,\g_2)$ is  also a twilled $3$-Lie algebra if and only if $H$ is a Maurer-Cartan element of the $L_\infty$-algebra $(\oplus_{m\geq0}C^m_{\Li}(\g_2,\g_1),l_1,l_2,l_3)$  given in Theorem \ref{quasi-as-shLie}.
\end{thm}
\begin{proof}
By Proposition \ref{thm:twist},  the twisting have the form:
\begin{eqnarray*}
\label{quasi-twisting-1}\hat{\mu}_1^{H}&=&\hat{\mu}_1,\\
\label{quasi-twisting-2}\hat{\psi}^{H}&=&\hat{\psi}+[\hat{\mu}_1,\hat{H}]_{\Li},\\
\label{quasi-twisting-3}\hat{\mu}_2^{H}&=&\hat{\mu}_2+[\hat{\psi},\hat{H}]_{\Li}+\half[[\hat{\mu}_1,\hat{H}]_{\Li},\hat{H}]_{\Li},\\
\label{quasi-twisting-4}\hat{\phi}_2^{H}&=&[\hat{\mu}_2,\hat{H}]_{\Li}+\half[[\hat{\psi},\hat{H}]_{\Li},\hat{H}]_{\Li}+\frac{1}{6}[[[\hat{\mu}_1,\hat{H}]_{\Li},\hat{H}]_{\Li},\hat{H}]_{\Li}.
\end{eqnarray*}
Thus, $((\huaG,\Omega^{H}),\g_1,\g_2)$ is  also a twilled $3$-Lie algebra if and only if $\hat{\phi}_2^{H}=0$, which implies that $H$ is a Maurer-Cartan element of the $L_\infty$-algebra $(\oplus_{m\geq0}C^m_{\Li}(\g_2,\g_1),l_1,l_2,l_3)$.
\end{proof}

\begin{cor}\label{Twisting-of-Twilled-3-Lie}
Let $(\huaG,\g_1,\g_2)$ be a twilled $3$-Lie algebra and $H:\g_2\longrightarrow\g_1$ a Maurer-Cartan element of the associated  $L_\infty$-algebra given in Theorem \ref{quasi-as-shLie}. Then
\begin{eqnarray*}
[u,v,w]_{H}:&=&[u,v,w]_2+[H(u),v,w]_2+[u,H(v),w]_2+[u,v,H(w)]_2\\
&&+[H(u),H(v),w]_2+[H(u),v,H(w)]_2+[u,H(v),H(w)]_2,\quad\forall u,v,w\in\g_2
\end{eqnarray*}
defines a $3$-Lie algebra structure on $\g_2$.
\end{cor}
\begin{proof}
By Corollary \ref{lem:twilled-t}, we deduce that $\hat{\mu}_2^{H}$ is a $3$-Lie algebra structure on $\huaG$. Since $||\hat{\mu}_2^{H}||=0|2$, we obtain that $\hat{\mu}_2^{H}$ gives rise to a $3$-Lie algebra structure on $\g_2$ as above.
\end{proof}

 \section{$\huaO$-operators and twilled $3$-Lie algebras} \label{sec:OT}

In this section, we use   \Ops~to construct   twilled-$3$-Lie algebras. In particular, we give examples of twilled 3-Lie algebras, which is not a generalized matched pair of 3-Lie algebras, to illustrate the differences between Lie algebras and 3-Lie algebras.

\begin{defi}{\rm (\cite{BGS-3-Bialgebras})}
Let $(\g,[\cdot,\cdot,\cdot]_{\g})$ be a $3$-Lie algebra and $(V;\rho)$ a representation of $\g$. A linear operator $T:V\rightarrow\g$ is called an {\bf \Op}~ on $(\g,[\cdot,\cdot,\cdot]_{\g})$ with respect to $(V;\rho)$ if $T$ satisfies
\begin{eqnarray}\label{eq-o-operator}
[Tu,Tv,Tw]_\g=T(\rho(Tu,Tv)w+\rho(Tv,Tw)u+\rho(Tw,Tu)v),\quad \forall u, v, w\in V.
\end{eqnarray}
\end{defi}

  Let $(\g,[\cdot,\cdot,\cdot]_\g)$ be a $3$-Lie algebra and $(V;\rho)$ its representation. Then $(\g\oplus V,[\cdot,\cdot,\cdot]_\ltimes)$ is a $3$-Lie algebra, where $[\cdot,\cdot,\cdot]_\ltimes$ is given by
   $$
   [x+u,y+v,z+w]_\ltimes=[x,y,z]_\g+\rho(x,y)w+\rho(y,z)u+\rho(z,x)v.
   $$
   This $3$-Lie algebra is called
   the semi-direct product $3$-Lie algebra and denoted by $\g\ltimes V$. It is obvious that $(\g\ltimes V,\g,V)$ is a strict twilled $3$-Lie algebra. We also use $\Omega $ to denote the $3$-Lie bracket $[\cdot,\cdot,\cdot]_\ltimes$. In this special case, $\Omega $ only contains $\hat{\mu}_1$, whose bidegree is $2|0$.

\begin{rmk}\label{rmk:OandMC}
   In \cite{THS}, a Lie $3$-algebra is constructed to characterize $\huaO$-operators (also called relative Rota-Baxter operator) on $3$-Lie algebras  as Maurer-Cartan elements. More precisely, $T:V\rightarrow\g$ is an {\Op}~ if and only if $$ [[[\hat{\mu}_1,\hat{T}]_{\Li},\hat{T}]_{\Li},\hat{T}]_{\Li}=0.$$
   It is straightforward to see that the $L_\infty$-algebra given in Theorem  \ref{quasi-as-shLie} reduces to the Lie $3$-algebra given in \cite{THS} when we consider the twilled $3$-Lie algebra $(\g\ltimes V,\g,V)$.
\end{rmk}

\Ops~ can also be characterized by twilled $3$-Lie algebras as follows.

\begin{thm}\label{Twilled-strong-3-Lie}
Let $(\g,[\cdot,\cdot,\cdot]_{\g})$ be a $3$-Lie algebra and $(V;\rho)$ a representation of $\g$. Then a linear map $T:V\rightarrow\g$ is an \Op ~ if and only if  $((\g\oplus V,\Omega^{T}),\g,V)$ is a twilled $3$-Lie algebra.
\end{thm}

\begin{proof}
By Proposition  \ref{thm:twist},  the twisting
$\Omega^{T}$ have the form:
\begin{eqnarray*}
\hat{\mu}_1^{T}&=&\hat{\mu}_1,\\
\hat{\psi}^{T}&=&[\hat{\mu}_1,\hat{T}]_{\Li},\\
\hat{\mu}_2^{T}&=&\half[[\hat{\mu}_1,\hat{T}]_{\Li},\hat{T}]_{\Li},\\
\hat{\phi}_2^{T}&=&\frac{1}{6}[[[\hat{\mu}_1,\hat{T}]_{\Li},\hat{T}]_{\Li},\hat{T}]_{\Li}.
\end{eqnarray*}
Therefore,  $((\g\oplus V,\Omega^{T}),\g,V)$ is a twilled $3$-Lie algebra if and only if
$\hat{\phi}_2^{T}=0,$ which is equivalent to that $T$ is an \Op.
\end{proof}

\begin{cor}
Let $T:V\rightarrow\g$ be an \Op~ on a $3$-Lie algebra $(\g,[\cdot,\cdot,\cdot]_{\g})$ with respect to a representation $(V;\rho)$.
Then the $\hat{\mu}_2^{T}$ defines a $3$-Lie algebra structure $[\cdot,\cdot,\cdot]_{V_{T}}$ on $V$ as well as a representation $\varrho:\wedge^{2}V\rightarrow \gl(\g)$ of the $3$-Lie algebra $(V,[\cdot,\cdot,\cdot]_{V_{T}})$ on the vector space $\g$ as follows
 \begin{eqnarray}
 [u,v,w]_{V_{T}}&=&\rho(Tu,Tv)w+\rho(Tv,Tw)u+\rho(Tw,Tu)v, \quad\forall u,v,w\in V,\\
  \varrho(u,v)x&=&[Tu,Tv,x]_{\g}-T(\rho(x,Tu)v)-T(\rho(Tv,x)u),\quad \forall x\in \g.
\end{eqnarray}
\end{cor}

\begin{proof}
By Corollary \ref{Twisting-of-Twilled-3-Lie} and Proposition \ref{Twilled-strong-3-Lie},
 we can obtain that $\hat{\mu}_2^{T}$ can give a $3$-Lie algebra structure on $V$, where for all $u,v,w\in V,$
 \begin{eqnarray*}
  [u,v,w]_{V_{T}}=\half[[\hat{\mu}_1,\hat{T}]_{\Li},\hat{T}]_{\Li}(u,v,w)=\rho(Tu,Tv)w+\rho(Tv,Tw)u+\rho(Tw,Tu)v,
 \end{eqnarray*}
 and $\varrho(u,v)x:=\hat{\mu}_2^{T}(u,v,x)$ is a representation of $(V,[\cdot,\cdot,\cdot]_{V_{T}})$ on the vector space $\g$, where
 \begin{eqnarray*}
\varrho(u,v)x=\half[[\hat{\mu}_1,\hat{T}]_{\Li},\hat{T}]_{\Li}(u,v,x)=[Tu,Tv,x]_{\g}-T(\rho(x,Tu)v)-T(\rho(Tv,x)u).
\end{eqnarray*}
The proof is finished.
\end{proof}

We further give the full description of the twilled 3-Lie algebra structure $\Omega^{T}$. For this purpose, we define $\nu:\g\lon\Hom(\wedge^2V,V)$ and $\tau:V\lon\Hom(\wedge^2\g,\g)$ by
\begin{eqnarray*}
\nu(x)(v,w)&=&\rho(Tw,x)v+\rho(x,Tv)w,\\
\tau(w)(x,y)&=&[Tw,x,y]_{\g}-T(\rho(x,y)w).
\end{eqnarray*}

\begin{cor}\label{Twilled-Multiplication-Omega}
With the above notations, we have
   \begin{eqnarray*}\label{Twilled-3-Lie-multiplication}
\Omega^{T}((x,u),(y,v),(z,w))&=&[x,y,z]_{\g}+\rho(x,y)w+\rho(y,z)u+\rho(z,x)v\\
\nonumber&&+\nu(x)(v,w)+\nu(y)(w,u)+\nu(z)(u,v)\\
\nonumber&&+ [u,v,w]_{V_{T}}+\varrho(u,v)x+\varrho(v,w)x+\varrho(w,u)y\\
\nonumber&&+\tau(w)(x,y)+\tau(u)(y,z)+\tau(v)(z,x).
\end{eqnarray*}
\end{cor}

\begin{rmk}
  Even though $\Omega^T$ is a twilled $3$-Lie algebra, however $(V;\rho,\nu)$ is not a generalized representation of $\g$ and $(\g;\varrho,\tau)$ is not a   generalized representation of $(V, [\cdot,\cdot,\cdot]_{V_{T}})$. Thus, $(\g,V;\rho,\nu,\varrho,\tau)$ is not a generalized matched pair of $3$-Lie algebras.
\end{rmk}

\begin{ex}{\rm
Consider the $3$-dimensional $3$-Lie algebra $(\g,[\cdot,\cdot,\cdot]_{\g})$ given with respect to a basis $\{e_1,e_2,e_3\}$  by
$$[e_1,e_2,e_3]_{\g}=e_1.$$
Thanks to Example 3.2 in {\rm \cite{THS}}, we obtain
\begin{eqnarray}\label{adjoint-representation-T-1}
T=\left(\begin{array}{ccc}
 1&0&0\\
 0&1&0\\
 0&0&-1
 \end{array}\right)
 \end{eqnarray}
 is an \Op~with respect to the adjoint representation.

By Proposition \ref{Twilled-strong-3-Lie}, $((\g\oplus \g',\Omega^{T}),\g,\g')$ is a $6$-dimensional $3$-Lie algebra, where $T$ is defined by \eqref{adjoint-representation-T-1} and $\g=\g'$.
 Let $\{\frke_1,\frke_2,\frke_3,\frke_4,\frke_5,\frke_6\}$ be the basis of $(\g\oplus \g',\Omega^{T})$, where
 $\frke_i=(e_i,0)\in \g, 1\leq i\leq 3$ and $\frke_j=(0,e_{j-3})\in \g', 4\leq j\leq 6.$
By Corollary \ref{Twilled-Multiplication-Omega},   the $3$-Lie bracket $\Omega^{T}$ is given by
\begin{alignat*}{4}
\Omega^{T}(\frke_1,\frke_2,\frke_3)&=\frke_1,&\quad\Omega^{T}(\frke_2,\frke_3,\frke_4)&=\frke_4,&\quad\Omega^{T}(\frke_1,\frke_2,\frke_6)&=-2\frke_1+\frke_4,&\quad \Omega^{T}(\frke_1,\frke_5,\frke_3)&=\frke_4,\\
\Omega^{T}(\frke_4,\frke_5,\frke_6)&=-\frke_4,&\quad
 \Omega^{T}(\frke_1,\frke_6,\frke_5)&=\frke_1,&\quad\Omega^{T}(\frke_3,\frke_4,\frke_5)&=2\frke_4-\frke_1,&\quad\Omega^{T}(\frke_4,\frke_2,\frke_6)&=-\frke_1.
\end{alignat*}
We only give the computation of the second equality. By Corollary \ref{Twilled-Multiplication-Omega}, we have
\begin{eqnarray*}
  \Omega^T(\frke_1,\frke_2,\frke_6)&=&\Omega^T((e_1,0),(e_2,0),(0,e_3))\\
                                   &=&(0,\rho(e_1,e_2)e_3)+(\tau(e_3)(e_1,e_2),0)\\
                                   &=&(0,e_1)+(-2e_1,0)\\
                                   &=&-2\frke_1+\frke_4.
\end{eqnarray*}
Using the notations in Corollary \ref{Twilled-Multiplication-Omega}, we have
\begin{eqnarray*}
  \rho(\frke_2,\frke_3)(\frke_4)=\frke_4,\quad \rho(\frke_1,\frke_2)(\frke_6)=\frke_4,\quad \rho(\frke_1,\frke_3)(\frke_5)=-\frke_4,\quad \nu(\frke_3)(\frke_4,\frke_5)&=2\frke_4
\end{eqnarray*}
and 
\begin{eqnarray*}
  \varrho(\frke_5,\frke_6)(\frke_1)=\frke_1,\quad \varrho(\frke_4,\frke_5)(\frke_3)=-\frke_1,\quad \varrho(\frke_4,\frke_6)(\frke_2)=\frke_1,\quad \tau(\frke_6)(\frke_1,\frke_2)&=-2\frke_1.
\end{eqnarray*}
It is straightforward to deduce that they do not give rise to a generalized matched pair.
}
\end{ex}

\begin{ex}{\rm
Consider the $4$-dimensional $3$-Lie algebra $(\g,[\cdot,\cdot,\cdot]_{\g})$ given with respect to a basis $\{e_1,e_2,e_3,e_4\}$  by
$$[e_2,e_3,e_4]_{\g}=e_1.$$
 Then $T=\left(\begin{array}{cccc}
 a_{11}&a_{12}&a_{13}&a_{14}\\
 a_{21}&a_{22}&a_{23}&a_{24}\\
 a_{31}&a_{32}&a_{33}&a_{34}\\
 a_{41}&a_{42}&a_{43}&a_{44}\\
 \end{array}\right)$ is an \Op~on $(\g,[\cdot,\cdot,\cdot]_{\g})$ with respect to the adjoint representation if and only if
$$
[Te_i,Te_j,Te_k]_{\g}=T([Te_i,Te_j,e_k]_{\g}+[Te_i,e_j,Te_k]_{\g}+[e_i,Te_j,Te_k]_{\g}),\quad\forall e_i,e_j,e_k\in\g.
$$
It is straightforward to deduce that $T$ is an \Op~if and only if
\begin{eqnarray*}
&&a_{22}a_{33}a_{44}-a_{22}a_{43}a_{34}-a_{23}a_{32}a_{44}+a_{24}a_{32}a_{43}+a_{23}a_{34}a_{42}-a_{24}a_{33}a_{42}\\
&=&(a_{22}a_{33}-a_{33}a_{32}-a_{43}a_{34}+a_{33}a_{44}-a_{24}a_{42}+a_{44}a_{22})a_{11}
\end{eqnarray*}
and
\begin{eqnarray*}
&&(a_{22}a_{33}-a_{33}a_{32}-a_{43}a_{34}+a_{33}a_{44}-a_{24}a_{42}+a_{44}a_{22})a_{21}\\
&=&(a_{22}a_{33}-a_{33}a_{32}-a_{43}a_{34}+a_{33}a_{44}-a_{24}a_{42}+a_{44}a_{22})a_{31}\\
&=&(a_{22}a_{33}-a_{33}a_{32}-a_{43}a_{34}+a_{33}a_{44}-a_{24}a_{42}+a_{44}a_{22})a_{41}\\
&=&0.
\end{eqnarray*}
In particular, we can obtain
\begin{eqnarray}\label{adjoint-representation-T-2}
T=\left(\begin{array}{cccc}
 1&0&0&0\\
 0&1&0&0\\
 0&0&1&0\\
 0&0&0&-1
 \end{array}\right)
 \end{eqnarray}
 is an \Op~with respect to the adjoint representation.

By Proposition \ref{Twilled-strong-3-Lie}, $((\g\oplus \g',\Omega^{T}),\g,\g)$ is a $8$-dimensional $3$-Lie algebra, where $T$ is defined by \eqref{adjoint-representation-T-2} and $\g=\g'$.
 Let $\{\frke_1,\frke_2,\frke_3,\frke_4,\frke_5,\frke_6,\frke_7,\frke_8\}$ be the basis of $(\g\oplus \g',\Omega^{T})$, where
 $\frke_i=(e_i,0)\in \g, 1\leq i\leq 4$ and $\frke_j=(0,e_{j-4})\in \g', 5\leq j\leq 8.$

By Corollary \ref{Twilled-Multiplication-Omega},   the $3$-Lie bracket $\Omega^{T}$ is given by
\begin{alignat*}{4}
 \Omega^{T}(\frke_2,\frke_3,\frke_4)&=\frke_1,&\quad\Omega^{T}(\frke_2,\frke_3,\frke_8)&=-2\frke_1+\frke_5,&\quad\Omega^{T}(\frke_2,\frke_7,\frke_4)&=\frke_5,&\quad\Omega^{T}(\frke_6,\frke_3,\frke_4)&=\frke_5,\\
 \Omega^{T}(\frke_6,\frke_7,\frke_8)&=-\frke_5,&\quad \Omega^{T}(\frke_6,\frke_7,\frke_4)&=2\frke_5-\frke_1,& \quad
 \Omega^{T}(\frke_3,\frke_6,\frke_8)&=\frke_1,&\quad \Omega^{T}(\frke_2,\frke_7,\frke_8)&=-\frke_1.
\end{alignat*}
}
\end{ex}

 \end{document}